\documentclass[11pt,a4,fleqn]{article}
\usepackage{graphicx}
\usepackage{amsmath,amssymb,latexsym,graphics,epsfig}
\usepackage{hyperref}
\usepackage{color}
\usepackage{amsthm}

\newtheorem{theorem}{\bf Theorem}[section]

\newtheorem{lemma}[theorem]{\bf Lemma}

\newtheorem{conjecture}[theorem]{\bf Conjecture}

\numberwithin{equation}{section}

\begin{document}
\title{{\Large Ramsey numbers of $5$-uniform loose  cycles}}

\author{ M. Shahsiah$^{\textrm{a},\textrm{b}}$ \\[2pt]
{\small $^{\textrm{a}}$Department of Mathematics, University of Khansar, Khansar, 87916-85163, Iran}\\[2pt]
{\small $^{\textrm{b}}$School of Mathematics, Institute for Research in Fundamental Sciences (IPM),}\\
{\small P.O. Box 19395-5746, Tehran, Iran }\\[2pt]
{shahsiah@ipm.ir}}

\date{}

\maketitle

\begin{abstract}
Gy\'{a}rf\'{a}s et al. determined the asymptotic value of the diagonal
 Ramsey number of  $\mathcal{C}^k_n$, $R(\mathcal{C}^k_n,\mathcal{C}^k_n),$  generating the same result for $k=3$ due to Haxell et al. Recently, the exact values of the Ramsey numbers of 3-uniform loose paths and cycles are completely determined. These results are motivations to conjecture that
for every $n\geq m\geq 3$ and $k\geq 3,$
$$R(\mathcal{C}^k_n,\mathcal{C}^k_m)=(k-1)n+\lfloor\frac{m-1}{2}\rfloor,$$
as mentioned by Omidi et al.
More recently, it is shown that this conjecture is true   for $n=m\geq 2$ and $k\geq 7$ and for $k=4$  when $n>m$ or $n=m$ is odd. Here we investigate this conjecture for $k=5$ and demonstrate that it holds for $k=5$ and  sufficiently large $n$.
%
%  We also demonstrate that
% this holds for $k=5,$ where $n\geq
%\lfloor\frac{3m}{2}\rfloor$.


\end{abstract}

%%%%%%%%%%%%%%%%%%%%%%%%%%%%%%%%%%%%%%%%%%%%%%%%%%%%%%%%%%%%%%%%%%%%%%%%
\section{\normalsize Introduction}

\bigskip
For two $k$-uniform hypergraphs  $\mathcal{G}$ and
$\mathcal{H},$
  the \textit{Ramsey number} $R(\mathcal{G},\mathcal{H})$ is
  the smallest  integer $N$ such that in every red-blue coloring of the
  edges of the complete $k$-uniform hypergraph $\mathcal{K}^k_N$ on $N$ vertices,  there is  a monochromatic copy of $\mathcal{G}$ in color red  or a monochromatic  copy of $\mathcal{H}$ in color blue.
 A {\it $k$-uniform  loose cycle} (shortly, a {\it cycle of length $n$}), denoted by  $\mathcal{C}_n^k,$   is a hypergraph with vertex set
$\{v_1,v_2,\ldots,v_{n(k-1)}\}$ and with the set of $n$ edges
$e_i=\{v_{(i-1)(k-1)+1},v_{(i-1)(k-1)+2},\ldots, v_{(i-1)(k-1)+k}\}$, $1\leq i\leq n$, where
we use mod $n(k-1)$ arithmetic.
%and adding a number $t$ to a set
%$H=\{v_1,v_2,\ldots, v_k\}$ means a shift, i.e. the set obtained
%by adding $t$ to subscripts of each element of $H$.
Similarly, a
{\it $k$-uniform  loose path} (shortly, a {\it path of length $n$}), denoted by  $\mathcal{P}_n^k,$  is a hypergraph with vertex set
$\{v_1,v_2,\ldots,v_{n(k-1)+1}\}$  and with the set of $n$ edges
$e_i=\{v_{(i-1)(k-1)+1},v_{(i-1)(k-1)+2},\ldots, v_{(i-1)(k-1)+k}\}$, $1\leq i\leq n$. For
an edge $e_i=\{v_{(i-1)(k-1)+1},v_{(i-1)(k-1)+2},\ldots, v_{i(k-1)+1}\}$ of a given loose path (also a given loose cycle)
$\mathcal{K}$, we denote by $f_{\mathcal{K},e_i}$ and $l_{\mathcal{K},e_i}$ the first  vertex  ($v_{(i-1)(k-1)+1}$) and the last vertex ($v_{i(k-1)+1}$) of $e_i,$ respectively.\\

%
% the first vertex ($v_{(i-1)(k-1)+1}$) and the last vertex ($v_{i(k-1)+1}$) are denoted by
%$f_{\mathcal{K},e_i}$ and
%$l_{\mathcal{K},e_i}$, respectively.\\
%In
%this paper, we consider the problem of finding the 2-color Ramsey
%number of $k$-uniform loose paths and cycles for
% $k=4,5$.\\
%For $k=2$ we get the usual definitions of a cycle $C_n$ and a path $P_n$ with $n$
%edges.\\

The problem  of determining or estimating Ramsey numbers is one of
the central problems in   combinatorics which has been of
interest to many investigators. In contrast to the graph case, there are relatively few results on the hypergraph Ramsey numbers.
%not much is on the Ramsey numbers of hypergraphs.
%Recently, this topic has received considerable attention.
The  Ramsey numbers of hypergraph loose
cycles were first considered
 by  Haxell et al. \cite{Ramsy number of
loose cycle}. They showed that
$R(\mathcal{C}^3_n,\mathcal{C}^3_n)$ is asymptotically
$\frac{5}{2}n$.
%More
%precisely, they proved that for all $\eta>0$ there exists
%$n_0=n_0(\eta)$ such that for every $n> n_0$, every $2$-coloring
%of $\mathcal{K}^{3}_{ 5(1+\eta)n/2}$ contains a monochromatic copy
%of $C_n^3$.
 Gy\'{a}rf\'{a}s et al.
%  S\'{a}rk\"{o}zy and Szemer\'{e}di
  \cite{Ramsy number of loose cycle for k-uniform} generalized this result  to
$k$-uniform loose cycles and showed that for $k\geq 3,$ $R(\mathcal{C}^k_n,\mathcal{C}^k_n)$ is asymptotic to $\frac{1}{2}(2k-1)n.$
%
% More precisely, they proved that for all
%$\eta>0$ there exists $n_0=n_0(\eta)$ such that for every $n> n_0,$ every 2-coloring
%of $\mathcal{K}^{k}_N$ with $N=(1+\eta)\frac{1}{2}(2k-1)n$
%contains a monochromatic copy of $\mathcal{C}^k_n.$\\

The investigation of the exact values of hypergraph loose paths and cycles was initiated by Gy\'{a}rf\'{a}s et al. \cite{subm},
%The first results on the exact values of loose paths and cycles were obtained  by Gy\'{a}rf\'{a}s and Raeisi \cite{subm},
 who   determined the exact
 values of
the Ramsey numbers of two $k$-uniform loose triangles and quadrangles. Recently, in \cite{The Ramsey number of loose paths  and loose cycles
in 3-uniform hypergraphs}, the authors completely determined the exact values of the Ramsey numbers of 3-uniform loose paths and cycles. More precisely,
they showed the following.
%Recently some interesting results were obtained on the exact values of the Ramsey numbers of loose paths and cycles of arbitrary lengths.
 %Gy\'{a}rf\'{a}s and Raeisi \cite{subm}  determined the exact
% values of
%the Ramsey numbers of two $k$-uniform loose triangles and quadrangles.
 %In \cite{The Ramsey number of loose paths  and loose cycles
%in 3-uniform hypergraphs}, the authors  proved  the following general result on the Ramsey number of $3$-uniform loose paths and cycles.

\begin{theorem}{\rm \cite{The Ramsey number of loose paths  and loose cycles
in 3-uniform hypergraphs}}\label{Omidi} For every $n\geq m\geq 3,$
\begin{eqnarray*}\label{5}
R(\mathcal{P}^3_n,\mathcal{P}^3_m)=R(\mathcal{P}^3_n,\mathcal{C}^3_m)=R(\mathcal{C}^3_n,\mathcal{C}^3_m)+1=2n+\Big\lfloor\frac{m+1}{2}\Big\rfloor.
\end{eqnarray*}
Moreover, for $n>m\geq 3,$ we have
\begin{eqnarray*}\label{5}
R(\mathcal{P}^3_m,\mathcal{C}^3_n)=2n+\Big\lfloor\frac{m-1}{2}\Big\rfloor.
\end{eqnarray*}
\end{theorem}
Regarding Ramsey numbers of $k$-uniform loose paths and cycles for $k\geq 3,$
 %They also posed the following conjecture in \cite{Ramsey numbers of loose cycles in uniform hypergraphs}, as mentioned in \cite{subm} by Gyarfas et al.
 in \cite{Ramsey numbers of loose cycles in uniform hypergraphs} the authors posed the following conjecture, as mentioned also in \cite{subm}.
%In \cite{Ramsey numbers of loose cycles in uniform hypergraphs},
%we presented another proof of Theorem \ref{Omidi} and posed the
%following conjecture.
%This gives a positive answer to a question of Gy\'{a}rf\'{a}s and
%Raeisi (see \cite{subm}).  Moreover, it has been shown that
%$R(\mathcal{P}^3_m,\mathcal{C}^3_n)=2n+\lfloor\frac{m-1}{2}\rfloor$
%for any  $n>m\geq 3$. Based on the above results we pose the following conjecture
%on the  value of the Ramsey number of loose cycles.
%%We believe that this lower bound is essentially the same as the
%upper bound. So we have the following Conjecture:

\begin{conjecture}\label{our conjecture}
Let $k\geq 3$ be an integer number. For every $n\geq m \geq 3$,
\begin{eqnarray*}\label{2}
R(\mathcal{P}^k_{n},\mathcal{P}^k_{m})=R(\mathcal{P}^k_{n},\mathcal{C}^k_{m})=R(\mathcal{C}^k_n,\mathcal{C}^k_m)+1=(k-1)n+\lfloor\frac{m+1}{2}\rfloor.
\end{eqnarray*}
\end{conjecture}

%The Case $k=3$ follows from Theorem \ref{Omidi}. The following lemma shows that the right hand side of
%(\ref{2}) is a lower bound for the Ramsey number of loose cycles.
% A similar Lemma is proved in \cite{subm} for $k$-uniform hypergraphs. But, for the sake of
% completeness, we include the proof here.
%\noindent Indeed, we showed that the conjecture is equivalent to
%have only the last equality.
%\noindent Also,  the following theorem is obtained on the Ramsey number of loose paths and cycles in $k$-uniform hypergraphs  \cite{Ramsey numbers of loose cycles in uniform
%hypergraphs}.
%
%
%\begin{theorem}{\rm \cite{Ramsey numbers of loose cycles in uniform hypergraphs}}\label{connection}
%Let $n\geq m\geq 2$ be given integers and
%$R(\mathcal{C}^k_{n},\mathcal{C}^k_{m})=(k-1)n+\lfloor\frac{m-1}{2}\rfloor.$
%Then
%$R(\mathcal{P}^k_n,\mathcal{C}^k_m)=(k-1)n+\lfloor\frac{m+1}{2}\rfloor$
%and
%$R(\mathcal{P}^k_n,\mathcal{P}^k_{m-1})=(k-1)n+\lfloor\frac{m}{2}\rfloor$.
%Moreover, for $n=m$ we have
%$R(\mathcal{P}^k_n,\mathcal{P}^k_m)=(k-1)n+\lfloor\frac{m+1}{2}\rfloor$.
%\end{theorem}

\noindent  They also showed that    Conjecture \ref{our conjecture} is equivalent to the following.

\begin{conjecture}\label{our conjecture2}
Let $k\geq 3$ be an integer number. For every $n\geq m \geq 3$,
\begin{eqnarray*}\label{2}
R(\mathcal{C}^k_n,\mathcal{C}^k_m)=(k-1)n+\lfloor\frac{m-1}{2}\rfloor.
\end{eqnarray*}
\end{conjecture}

More recently, it is   shown that Conjecture \ref{our conjecture2} holds for $n=m$ and $k\geq 7$ (see \cite{diagonal}). For small values of $k$,
in  \cite{4-uniform}, the authors demonstrate that
Conjecture \ref{our conjecture2}
holds for $k=4$ where
$n>m$ or $n=m$ is odd. Therefore, in this regard, investigating  the small cases $k=5,6$  are  interesting.
In this paper, we focus on the case $k=5$ and
%our aim is  study of Ramsey numbers on $5$-uniform loose cycles and
shall show that
%In this paper, we show that for $k=5,$
 Conjecture \ref{our conjecture2}  holds for sufficiently
large $n$.
More precisely, we prove that Conjecture \ref{our conjecture2} holds  for $k=5$ when
$n\geq\lfloor\frac{3m}{2}\rfloor$. We remark that in the proof we extend the method that used in \cite{Ramsey numbers of loose cycles in uniform hypergraphs} and use a modified version of some lemmas in  \cite{4-uniform}.
 Throughout the paper, by Lemma 1 of
\cite{subm}, it  suffices to
 prove only the  upper bound for the claimed Ramsey numbers.
  Throughout the paper, for
a 2-edge colored hypergraph $\mathcal{H}$ we denote by
$\mathcal{H}_{\rm red}$ and $\mathcal{H}_{\rm blue}$ the induced
hypergraphs on red edges and blue edges, respectively.
%Also we
%denote by $|\mathcal{H}|$ and $\|\mathcal{H}\|$ the number of
%vertices and edges of $\mathcal{H}$, respectively.
%%%%%%%%%%%%%%%%%%%%%%%%%%%%

%%%%%%%%%%%%%%%%%%%%%%%%%%%%

%%%%%%%%%%%%%%%%%%%%%%%%%%%%%%%%%%%%%%%%%%%%%%%%%%%%%%%%%%%%%%%%%%%%%%%%%%%%%%%%%%%%%%%%%%%%%%%%%%%%%%%%%%%
\section{\normalsize Preliminaries}

\medskip
In this section, we prove  some  lemmas that will be used in the follow up section.
%we prove  some  lemmas that will be needed in our main results.
% whose proofs are
%based on the property of monochromatic maximal paths in 3-uniform hypergraphs.
 Also, we recall the following result  from  \cite{Ramsey numbers of loose cycles in uniform hypergraphs}.
%
% \begin{theorem}\label{R(Pk3,Pk3)} {\rm \cite{subm}}
%For every $k\geq 3$,
%\begin{itemize}
%\item [\rm{(a)}]~
%$R(\mathcal{P}^k_3,\mathcal{P}^k_3)=R(\mathcal{C}^k_3,\mathcal{P}^k_3)=R(\mathcal{C}^k_3,\mathcal{C}^k_3)+1=3k-1$,
%\item [\rm{(b)}]~
%$R(\mathcal{P}^k_4,\mathcal{P}^k_4)=R(\mathcal{C}^k_4,\mathcal{P}^k_4)=R(\mathcal{C}^k_4,\mathcal{C}^k_4)+1=4k-2$.
%\end{itemize}
%\end{theorem}

 \begin{theorem}\label{R(C3,C4)}{\rm \cite{Ramsey numbers of loose cycles in uniform hypergraphs}}
Let $n,k\geq 3$ be  integer numbers. Then
\begin{eqnarray*}R(\mathcal{C}^k_3,\mathcal{C}^k_n)= (k-1)n+1.\end{eqnarray*}
\end{theorem}

%%%%%%%%%%%%%%%%%%%%%%%%%%%%

%%%%%%%%%%%%%%%%%%%%%%%%%%%%
\vspace{0.5 cm}
Before we state our main results we need some definitions. Let
$\mathcal{H}$ be a 2-edge colored complete $5$-uniform hypergraph,
$\mathcal{P}$ be a loose path in $\mathcal{H}$ and $W$ be a set of
vertices with $W\cap V(\mathcal{P})=\emptyset$. By a {\it
$\varpi_S$-configuration}, we mean a copy of $\mathcal{P}^5_2$
with edges $$\{x,a_1,a_2,a_3,a_4\},
\{a_4,a_5,a_6,a_7,y\},$$ so that $\{x,y\}\subseteq W$ and  $S=\{a_j : 1\leq j
\leq 7\}\subseteq (e_{i-1}\setminus \{f_{\mathcal{P},e_{i-1}}\})\cup e_{i}
\cup e_{i+1}\cup e_{i+2}$ is a set of unordered vertices of
 $4$
 %(resp. of two)
 consecutive edges of
$\mathcal{P}$ with  $|S\cap (e_{i-1}\setminus \{f_{\mathcal{P},e_{i-1}}\})|\leq 1.$ Note that it is possible to have $S\subseteq e_i\cup e_{i+1}\cup e_{i+2},$ this case happens only when $S\cap (e_{i-1}\setminus \{f_{\mathcal{P},e_{i-1}}\})\subseteq \{f_{\mathcal{P},e_{i}}\}$.
%(resp. for $k=3$)
 The vertices $x$ and $y$ are called {\it
the end vertices} of this configuration. A
$\varpi_{S}$-configuration with $S\subseteq (e_{i-1}\setminus \{f_{\mathcal{P},e_{i-1}}\})\cup e_{i}
\cup e_{i+1}\cup e_{i+2}$,
%(resp. $S\subseteq e_i\cup e_{i+1}$)
 is {\it good} if at least one of the vertices of
$e_{i+2}\setminus e_{i+1}$
%(resp. $e_{i+1}\setminus e_{i}$)
is not in $S$. We say that a monochromatic path
$\mathcal{P}=e_1e_2\ldots e_n$ is {\it maximal with respect to}
(w.r.t. for short) $W\subseteq V(\mathcal{H})\setminus
V(\mathcal{P})$ if there is no $W'\subseteq W$ so that for some
$1\leq r\leq n$ and some $1\leq i \leq n-r+1,$ the path
\begin{eqnarray*}
\mathcal{P}'= \left\lbrace
\begin{array}{lr}
 e'_1e'_{2}\ldots e'_{n+1}, &i=1 \ \ {\rm and}\ \  r=n,\vspace{.2 cm}\\
  e'_1e'_{2}\ldots e'_{r+1}e_{r+1}\ldots e_n, &i=1 \ \ {\rm and}\ \  1\leq r<n,\vspace{.2 cm}\\
e_1\ldots e_{i-1}e'_ie'_{i+1}\ldots e'_{i+r}e_{i+r}\ldots e_n,\ \  & 2\leq i \leq n-r,\vspace{.2 cm}\\
e_1e_2\ldots e_{i-1}e'_ie'_{i+1}\ldots e'_{n+1}
&2 \leq i=n-r+1.
\end{array}
\right.\vspace{.2 cm}
\end{eqnarray*}
%
%either $\mathcal{P}'=e_1e_2\ldots e_{i-1}e'_ie'_{i+1}\ldots e'_{n+1}$ (for $i=n-r+1$) or $\mathcal{P}'=e_1e_2\ldots e_{i-1}e'_ie'_{i+1}\ldots e'_{i+r}e_{i+r}\ldots e_n$ (for some $1\leq i\leq n-r+1$)
 %\begin{eqnarray*}\mathcal{P}'=e_1e_2\ldots e_{i-1}e'_ie'_{i+1}\ldots e'_{i+r}e_{i+r}\ldots e_n,\end{eqnarray*}
 is a monochromatic path with $n+1$ edges and
 the following properties:

 \begin{itemize}
\item[(i)] $V(\mathcal{P}')=V(\mathcal{P})\cup W'$,
 \item[(ii)] if $i=1$, then
$f_{\mathcal{P}',e'_1}=f_{\mathcal{P},e_1}$,
\item[(iii)] if
 $i=n-r+1$, then
$l_{\mathcal{P}',e'_{n+1}}=l_{\mathcal{P},e_n}$.
\end{itemize}
In the other words, we say a monochromatic path $\mathcal{P}$ is maximal w.r.t. $W\subseteq V(\mathcal{H})\setminus
V(\mathcal{P})$, if one   can not extend  $\mathcal{P}$ any longer with the same color  using the vertices of $W.$
 Clearly, if $\mathcal{P}$ is maximal w.r.t. $W$, then it is maximal w.r.t. every
$W'\subseteq W$ and also every loose path $\mathcal{P}'$ which is a sub-hypergraph of $\mathcal{P}$ is again maximal w.r.t. $W$.\\

%  same color to $\mathcal{P}$ in $\mathcal{H}$ where $V(\mathcal{P}')=V(\mathcal{P})\cup W'$,
%   $v_{\mathcal{P}',e'_i}=v_{\mathcal{P},e_i}$ for $i=1$ and $\hat{v}_{\mathcal{P}',e'_{i+r}}=\hat{v}_{\mathcal{P},e_n}$
   % for $i+r-1=n$.

   The following lemma is indeed the modified version of  \cite{4-uniform, {Lemma 2.3}} for $5$-uniform hypergraphs. But, for the sake of completeness, we give a proof here.
   %\noindent   We use these definitions to deduce the following  essential lemma.

\bigskip
\begin{lemma}\label{spacial configuration2}
Assume that   $\mathcal{H}=\mathcal{K}^5_{n},$  is
$2$-edge colored red and blue. Let $\mathcal{P}\subseteq \mathcal{H}_{\rm red}$ be a maximal path w.r.t. $W,$
where $W\subseteq V(\mathcal{H})\setminus V(\mathcal{P})$ and  $|W|\geq 5$.
Let $A_1=\{f_{\mathcal{P},e_{1}}\}=\{v_1\}$ and
$A_i=e_{i-1}\setminus\{f_{\mathcal{P},e_{i-1}}\}$ for $i>1$.
% $W\subseteq V(\mathcal{H})\setminusV(\mathcal{P})$ with $|W|\geq k$.
 Then for every three  consecutive edges $e_i,e_{i+1}$ and $e_{i+2}$
of $\mathcal{P}$ and for each $u\in A_i$ there  is a
 good $\varpi_S$-configuration, say
$C=fg$, in $\mathcal{H}_{\rm blue}$ with  end vertices  $x\in f$ and $y\in g$  in $W$ and
 \begin{eqnarray*}S\subseteq \Big((e_i\setminus
\{f_{\mathcal{P},e_{i}}\})\cup \{u\}\Big)\cup e_{i+1} \cup \Big(e_{i+2}\setminus \{v\}\Big),\end{eqnarray*} for
some $v\in A_{i+3}$. Moreover, there are two subsets $W_1\subseteq W$ and $W_2\subseteq W$ with $|W_1|\geq |W|-3$ and $|W_2|\geq |W|-4$ so that for every distinct vertices $x'\in W_1$ and $y'\in W_2$, the path $C'=\Big((f\setminus\{x\})\cup\{x'\}\Big)\Big((g\setminus\{y\})\cup\{y'\}\Big)$ is also a good $\varpi_S$-configuration in $\mathcal{H}_{\rm blue}$ with  end vertices $x'$ and $y'$ in $W.$
\end{lemma}
\begin{proof}{ Let $\mathcal{P}=e_1e_2\ldots e_m\subseteq \mathcal{H}_{\rm red}$ be a maximal path w.r.t. $W\subseteq V(\mathcal{H})\setminus V(\mathcal{P})$, where
\begin{eqnarray*}
e_i=\{v_{4i-3},v_{4i-2},v_{4i-1},v_{4i},v_{4i+1}\}, \hspace{1 cm} i=1,2,\ldots, m.
\end{eqnarray*}
%\begin{eqnarray*}
%$$e_1=\{v_1,v_2,v_3,v_4\},  e_2=\{v_4,v_5,v_6,v_7\}.$$
 % \end{eqnarray*}
Suppose that $e_i,e_{i+1}$ and $e_{i+2}$ are three consecutive edges of $\mathcal{P}$ and $u\in A_i$. (Note that for $i=1,$ we have $u=v_1$)
 Among  different choices of $4$ distinct vertices of $W,$ choose a $4$-tuple
$X=(x_1,x_2,x_3,x_4)$ so that $E_X$ has the  minimum
number of  blue edges, where $E_X=\{f_1,f_2,f_3,f_4\}$ and
\begin{eqnarray*}
&&f_{1}=\{u,x_1,v_{4i-2},v_{4i+2},v_{4i+6}\},\\
&&f_{2}=\{v_{4i-2},x_2,v_{4i-1},v_{4i+3},v_{4i+7}\},\\
&&f_{{3}}=\{v_{4i-1},x_{3},v_{4i},v_{4i+4},v_{4i+8}\},\\
&&f_{{4}}=\{v_{4i},x_{4},v_{4i+1},v_{4i+5},v_{4i+9}\}.
\end{eqnarray*}
Note that for $1\leq k \leq 4,$ we have $|f_k\cap(e_{i+2}\setminus\{f_{\mathcal{P},e_{i+2}}\})|=1.$
 Since $\mathcal{P}$ is a maximal path w.r.t.   $W,$ there is $1\leq j\leq
4$ so that  the edge $f_{j}$ is blue. Otherwise, replacing $e_ie_{i+1}e_{i+2}$ by $f_1f_2f_3f_4$ in $\mathcal{P}$ yields a red path $\mathcal{P}'$ with $n+1$ edges; this is a contradiction.
 Let $W_1=(W\setminus\{x_1,x_2,x_{3},x_4\})\cup\{x_j\}$.
For each vertex
$x\in W_1$  the edge
$f_x=(f_j\setminus\{x_j\})\cup\{x\}$ is blue. Otherwise, the number of blue edges in $E_Y$
% $Y=(x_1,x_2,\ldots,x_{j-1},x,x_{j+1},x_{k-1}),$
 is less than this number for $E_X$, where $Y$ is obtained from $X$  by replacing  $x_j$ to $x$. This is a contradiction.\\

Now we choose $h_1,h_2,h_3,h_4$ as follows. If $j=1,$ then set
\begin{eqnarray*}
&&h_1=\{u,v_{4i+3},v_{4i+7},v_{4i-1}\},h_2=\{v_{4i-1},v_{4i-2},v_{4i+4},v_{4i}\},\\
&&h_3=\{v_{4i},v_{4i+2},v_{4i+8},v_{4i+1}\},h_4=\{v_{4i+1},v_{4i+5},v_{4i+6},v_{4i+9}\}.
\end{eqnarray*}
If $j=2,$ then set
\begin{eqnarray*}
&&h_1=\{u,v_{4i-2},v_{4i+1},v_{4i}\},h_2=\{v_{4i},v_{4i-1},v_{4i+6},v_{4i+2}\},\\
&&h_3=\{v_{4i+2},v_{4i+3},v_{4i+8},v_{4i+4}\},h_4=\{v_{4i+4},v_{4i+5},v_{4i+7},v_{4i+9}\}.
\end{eqnarray*}
If $j=3,$ then set
\begin{eqnarray*}
&&h_1=\{u,v_{4i-2},v_{4i-1},v_{4i+1}\},h_2=\{v_{4i+1},v_{4i},v_{4i+6},v_{4i+2}\},\\
&&h_3=\{v_{4i+2},v_{4i+4},v_{4i+7},v_{4i+3}\},h_4=\{v_{4i+3},v_{4i+5},v_{4i+8},v_{4i+9}\}.
\end{eqnarray*}
If $j=4,$ then set
\begin{eqnarray*}
&&h_1=\{u,v_{4i-2},v_{4i},v_{4i-1}\},h_2=\{v_{4i-1},v_{4i+1},v_{4i+6},v_{4i+2}\},\\
&&h_3=\{v_{4i+2},v_{4i+5},v_{4i+7},v_{4i+3}\},h_4=\{v_{4i+3},v_{4i+4},v_{4i+8},v_{4i+9}\}.
\end{eqnarray*}
Note that in each the above cases, for $1\leq k \leq 4,$ we have  $|h_k\cap (f_j\setminus\{x_j\})|=1$ and $|h_k\cap(e_{i+2}\setminus(f_j\cup\{f_{\mathcal{P},e_{i+2}}\}))|\leq 1$.
Let $Y=(y_1,y_2,y_{3},y_4)$ be a $4$-tuple of distinct vertices of $W\setminus\{x_j\}$ with minimum number of blue edges in $F_{Y}$,  where $F_Y=\{g_1,g_2,g_3,g_4\}$ and $g_k=h_k\cup\{y_k\}$ for $1\leq k \leq 4$.
 Again since $\mathcal{P}$ is maximal w.r.t. $W$, for some $1\leq \ell\leq 4$ the edge
$g_{\ell}$ is blue and also, for each vertex $y_a\in W_2=
(W\setminus\{x_j,y_{1},y_{2},y_3,y_4\})\cup \{y_\ell\}$ the edge $g_a=(g_{\ell}\setminus\{y_{\ell}\})\cup\{y_a\}$ is blue. Set $f=f_j$ and $g=g_{\ell}$. Clearly $C=fg$ is our desired configuration with end vertices $x_j\in f$ and $y_{\ell}\in g$ in $W.$ Moreover,
 for distinct vertices $x'\in W_1$ and $y'\in W_2$, the path $C'=\Big((f_j\setminus\{x_j\})\cup\{x'\}\Big)\Big((g_{\ell}\setminus\{y_{\ell}\})\cup\{y'\}\Big)$ is also a good $\varpi_S$-configuration in $\mathcal{H}_{\rm blue}$ with  end vertices $x'$ and $y'$ in $W.$
% we have $C=fg$ which is our desired configuration, where $f=(f_j\setminus\{x_j\})\cup\{x'\}$ and $g=(g_{\ell}\setminus\{y_{\ell}\})\cup\{y'\}.$
  Since $|W_1|=|W|-2$, each vertex of $W,$ with the exception of at most $2,$ can be considered as an end vertex of $C'.$
  Note that the
configuration $C$ (and also $C'$) contains at most two vertices of $e_{i+2}\setminus
e_{i+1}$.\\}\end{proof}

Also, we need the following lemma.

\begin{lemma}\label{there is P3:2}
Let $\mathcal{H}=\mathcal{K}^5_{q}$ be $2$-edge colored red and blue
and $\mathcal{P}=e_1e_2\ldots e_n \subseteq \mathcal{H}_{\rm
blue}$ be a maximal path w.r.t. $W,$ where $W\subseteq
V(\mathcal{H})\setminus V(\mathcal{P})$ and $|W|\geq 7$. Assume
that  $A_1=\{f_{\mathcal{P},e_1}\}$ and
$A_i=V(e_{i-1})\setminus\{f_{\mathcal{P},e_{i-1}}\}$ for $i>1$.
Then for every two consecutive edges $e_i$ and $e_{i+1}$ of
$\mathcal{P}$ and for each $u\in A_i$, there is a
$\mathcal{P}_3^5\subseteq \mathcal{H}_{\rm red}$, say
$\mathcal{Q}$, with end vertices in $W$ such that
$V(\mathcal{Q})\subseteq ((e_i\setminus\{f_{\mathcal{P},e_i}\})\cup\{u\})\cup e_{i+1}\cup W,$
%for some $W'\subseteq W$ with $|W'|\leq 6$ and
   at least one of the vertices
of  $A_{i+2}$  is not in $\mathcal{Q}$
 and $\vert W\cap
 V(\mathcal{Q})\vert \leq 6$.
Moreover, each vertex of $W,$ with
the exception of at most one, can be considered as an end vertex
of  $\mathcal{Q}$.
\end{lemma}

%%%%%%%%%%%%%%%%%%%%%%%%%%%%
\begin{proof}{ Let $\mathcal{P}=e_1e_2\ldots e_n \subseteq \mathcal{H}_{\rm
blue}$ be a maximal path w.r.t. $W\subseteq V(\mathcal{H})\setminus V(\mathcal{P}),$ where
\begin{eqnarray*}
e_i=\{v_{4i-3},v_{4i-2},v_{4i-1},v_{4i},v_{4i+1}\},\hspace{1 cm} 1\leq i\leq n.
\end{eqnarray*}
Also, let  $e_i=\{v_{4i-3},v_{4i-2},v_{4i-1},v_{4i},v_{4i+1}\}$ and
$e_{i+1}=\{v_{4i+1},v_{4i+2},v_{4i+3},v_{4i+4},v_{4i+5}\}$ be two
consecutive edges of $\mathcal{P},$ $u\in A_i$ and  $W=\{x_1,...,x_t\}$ (note that for $i=1$, we have $u=v_1$).
%and $T=\{ 1,2,\cdots,t\}$.
\bigskip

\noindent \textbf{Case 1. }There exist   $x_{j},x_{j'}\in W$
 such that the edge
$e=\{u,v_{4i-2},v_{4i-1},x_j,x_{j'}\}$ is blue.

\medskip
\noindent Suppose without loss of generality that  $x_j=x_1$ and
$x_{j'}=x_2$. If for every vertex $x_3\in W\setminus\{x_1,x_2\}$ the
edge $f_1=\{x_2,x_3,v_{4i+1},v_{4i+3},v_{4i+4}\}$ is blue, then
for arbitrary vertices $x_4,x_5,x_6,x_7\in W\setminus\{x_1,x_2,x_3\}$ the edges
\begin{eqnarray*}
&&f_2=\{x_1,x_{4},v_{4i+4},v_{4i+5},v_{4i+2}\},\\
&&f_3=\{v_{4i+2},v_{4i-2},x_{3},x_{5},v_{4i}\},\\
&&f_4=\{v_{4i},v_{4i-1},v_{4i+1},x_{6},x_{7}\},
\end{eqnarray*}
are red (since $\mathcal{P}$ is maximal w.r.t. $W$) and $\mathcal{Q}=f_2f_3f_4$ is the desired path. Note that for every
$\alpha\neq 2,$ $x_{\alpha}$ can be considered as an end vertex of
$\mathcal{Q}$. To see that, let $x_3 \notin
\{x_{\alpha},x_1,x_2\}$. Thereby, for $\alpha=1$ we have
$x_{\alpha}\in f_2$ and otherwise we may assume that
$x_{\alpha}=x_6$ ($x_{\alpha}\in f_4$).

\noindent Now, we may assume that there is a vertex $x_3\in W\setminus \{x_1,x_2\}$
so that $f_1$ is red. Again, since $\mathcal{P}$ is maximal w.r.t. $W,$ the path
$\mathcal{Q}=f_1f_4\{v_{4i-2},v_{4i},v_{4i+2},x_{4},x_5\}$ is
 the desired red $\mathcal{P}^5_3$. It is easy to check that for every
$\alpha\neq 1,$ $x_{\alpha}$ can be considered as  an end vertex of
$\mathcal{Q}.$
%(for $\alpha\in \{2,3\}$, $x_{\alpha}\in f_1$ and
%otherwise w.l.g. let $x_{\alpha}=x_4$).

\bigskip
\noindent \textbf{Case 2. }There exist   $x_{j},x_{j'}\in W$
 such that the edge $e'=\{v_{4i-1}, v_{4i}, v_{4i+1}, x_j,
x_{j'}\}$ is blue.

\medskip
\noindent Similar to Case 1, assume that $x_j=x_1$ and $x_{j'}=x_2$.
Since $\mathcal{P}$ is maximal w.r.t. $W,$ for every
%$x_{3},x_{4},x_{5},x_{6},x_{7}$
$x_{\ell} \in W\setminus \{x_1,x_2\},$ $3\leq {\ell}\leq 7,$ the edges
$g_1=\{u,v_{4i-2},v_{4i-1},x_{3},x_{4}\}$ and
$g_2=\{u,v_{4i},x_{5},x_{6},x_{7}\}$ are red (if the edge $g_{\ell},$ $1\leq {\ell} \leq 2,$ is blue, replacing the edge $e_i$ by $g_{\ell}e'$ in the path $\mathcal{P},$ yields a blue path $\mathcal{P}'$ with $n+1$ edges, a contradiction). If there is
$k\in \{1,2\}$ such that the edge
$g_3=\{v_{4i+3},v_{4i+4},v_{4i+5},x_k,x_{5}\}$ is blue, then for
   $\ell\in \{1,2\}\setminus\{k\},$ the edge
$g_4=\{v_{4i+1},v_{4i+2},v_{4i+3},x_{6},x_{\ell}\}$ is red. So
$\mathcal{Q}=g_1 g_2 g_4$ makes the  desired red $\mathcal{P}^5_3$ (it
is obvious that for every $\alpha \neq k,$ $x_{\alpha}$ can be
seen as an end vertex of $\mathcal{Q}$). So we may assume that for
every $k\in \{1,2\}$, the edge $g_3$ is red and $\mathcal{Q}=g_1 g_2 g_3$
is a red $\mathcal{P}^5_3$ such that $v_{4i+2}\notin V(\mathcal{Q})$
and every $x_{\alpha}\in W$ can be considered as an end vertex of
$\mathcal{Q}$.

\bigskip \noindent \textbf{Case 3. }For every  $x_1,x_{2},x_3,x_{4}\in W,$  the edges $e=\{u, v_{4i-2}, v_{4i-1}, x_{1}, x_{2}\}$ and
$e'=\{v_{4i-1},v_{4i},v_{4i+1},x_{3},x_{4}\}$ are red.

\medskip
\noindent
We may assume that for every $k\neq j$, $1\leq j\leq 4$,
the edge $h_1=\{x_{4}, v_{4i+3}, v_{4i+4}, v_{4i+5}, x_{\ell}\}$ is red
and so $\mathcal{Q}=ee'h_1$ is the desired path (note that $v_{4i+2}\notin V(\mathcal{Q})$). If
not, since $\mathcal{P}$ is maximal w.r.t. $W,$
for every vertices $x_{\ell}, x_{\ell'} \in W\setminus\{x_1,$ $x_2,$ $x_3,$
$x_4,$ $x_k\}$ the edge $h_2=\{v_{4i+1},v_{4i+2},v_{4i+3}, x_{\ell},
x_{\ell'}\}$ is red and
 $\mathcal{Q}=ee'h_2$ is the  desired  path (note that, in this case, every vertex of
$W$ can be considered as an end vertex of $\mathcal{Q}$).

}\end{proof}

Consider a given  2-edge colored complete 5-uniform hypergraph $\mathcal{H}$.
 By Lemma \ref{spacial configuration2}, we can find many disjoint blue $\varpi_S$-configuration corresponding to a maximal
red loose path. The following lemma guarantees how we can connect these configurations
 to make at most two blue paths $\mathcal{Q}$ and $\mathcal{Q}'$ so that $\|\mathcal{Q}\cup \mathcal{Q}'\|$ is sufficiently large.\\

\begin{lemma}\label{there is a Pl}
Let   $\mathcal{H}=\mathcal{K}_l^5$  be two edge
colored red and blue. Also let $\mathcal{P}=e_1e_2\ldots e_n,$ $n\geq
3,$ be a maximal red path w.r.t. $W,$ where $W\subseteq
V(\mathcal{H})\setminus V(\mathcal{P})$ and  $|W|\geq 5$. Then for
some $r\geq 0$ and $W'\subseteq W$ there are  two disjoint blue
paths $\mathcal{Q}$ and $\mathcal{Q}',$ with $\|\mathcal{Q}\|\geq 2$ and
\begin{eqnarray*}
\|\mathcal{Q}\cup \mathcal{Q}'\|=2(n-r)/3=\left\lbrace
\begin{array}{ll}
2(|W'|-2)  &\mbox{if} \ \|\mathcal{Q}'\|\neq 0,\vspace{.5 cm}\\
2(|W'|-1) & \mbox{if}\ \|\mathcal{Q}'\|=0,
\end{array}
\right.\vspace{.2 cm}
\end{eqnarray*}
 between $W'$ and
$\overline{\mathcal{P}}=e_1e_2\ldots e_{n-r}$ so that $e\cap W'$
is actually the end vertex of $e$ for each edge $e\in
\mathcal{Q}\cup \mathcal{Q}'$ and at least one of the vertices of
$e_{n-r}\setminus e_{n-r-1}$ is not in $V(\mathcal{Q})\cup
V(\mathcal{Q}')$. Moreover, if $\|\mathcal{Q}'\|=0$ then either
$x=|W\setminus W'|\in\{2,3\}$ or $x\geq 4$ and $0\leq r \leq
2$. Otherwise, either $x=|W\setminus W'|=1$ or $x\geq 2$ and
$0\leq r \leq 2$.

%for every $W\subseteq V(\mathcal{H})\setminus V(\mathcal{P})$.
%Then there is a blue path $Q$ of length
%$t=min\{2\lfloor\frac{n}{2} \rfloor, 2|W|-??\}$ between
%$e_1e_2\ldots e_t$ and $W$ so that the end vertices of $Q$ are in
%$W$ and each vertex of $W$, except end vertices, lies on two edges
%of $Q$. Moreover the end vertices of $Q$ ???.
\end{lemma}

%%%%%%%%%%%%%%%%%%%%%%%%%%%%
\begin{proof}{
Let $\mathcal{P}=e_1e_2\ldots e_{n}$ be a maximal  red path w.r.t.
$W,$ $W\subseteq V(\mathcal{H})\setminus V(\mathcal{P})$, and
\begin{eqnarray*} e_i=\{v_{(i-1)(k-1)+1},v_{(i-1)(k-1)+2},\ldots,v_{i(k-1)+1}\}, \hspace{1 cm} i=1,2,\ldots,n,\end{eqnarray*}
are the edges of $\mathcal{P}$.\\\\
{\bf Step 1:} Set  $\mathcal{P}_1=\mathcal{P}$, $W_1=W$ and
$\overline{\mathcal{P}}_1=\mathcal{P}'_1=e_{1}e_{2}e_3$. Since $\mathcal{P}$ is
maximal w.r.t. $W_1$, using Lemma
 \ref{spacial configuration2} there is a good $\varpi_S$-configuration, say
 $\mathcal{Q}_1=f_1g_1,$ in $\mathcal{H}_{\rm blue}$  with end vertices $x\in f_1$ and $y\in g_1$ in $W_1$
 so that $S\subseteq \mathcal{P}'_1$  and
 % between $\mathcal{P}'_1$ and $W'_1\subseteq W_1$
  $\mathcal{Q}_1$ does not contain  a vertex of
 $e_3\setminus e_2,$ say $u_1$. Set  $X_1=|W\setminus V(\mathcal{Q}_1)|$,
$\mathcal{P}_2=\mathcal{P}_1\setminus \overline{\mathcal{P}}_1=e_4e_5\ldots e_n$ and
$W_2=W.$
%(W_1\setminus V(\mathcal{Q}_1))\cup\{x_1,y_1\}$.
 If  $|W_2|=5$ or $\|\mathcal{P}_2\|\leq 2$, then $\mathcal{Q}=\mathcal{Q}_1$ is a blue path between $W'=W_1\cap V(\mathcal{Q}_1)$
 and $\overline{\mathcal{P}}=\overline{\mathcal{P}}_1$ with desired properties. Otherwise, go to Step
 2.\\\\
%If
%$\mathcal{P}'_1=\mathcal{P}$,
 %then $\mathcal{Q}=\mathcal{Q}_1$ is a blue path between $W'=W'_1$ and $\overline{\mathcal{P}}=\mathcal{P}$ with desired properties.

\noindent {\bf Step 2:}
 Clearly  $|W_2|\geq 6$ and $\|\mathcal{P}_2\|\geq 3.$
 Set $\overline{\mathcal{P}}_2=e_4e_5e_6$ and
 $\mathcal{P}'_2=((e_{4}\setminus \{f_{\mathcal{P},e_{4}}\})\cup\{u_1\})e_5e_6$.
 Since $\mathcal{P}$ is maximal w.r.t. $W_2$,
 using Lemma \ref{spacial configuration2} there is  a good $\varpi_S$-configuration, say
 $\mathcal{Q}_2=f_2g_2,$ in $\mathcal{H}_{\rm blue}$ with end vertices $x\in f_2$ and $y\in g_2$ in $W_2$
% $\mathcal{P}'_2$ and $W'_2\subseteq W_2$
 such that $S\subseteq \mathcal{P}'_2$  and  $\mathcal{Q}_2$ does not contain a vertex of $e_6\setminus e_5,$ say $u_2.$ By Lemma \ref{spacial configuration2}, there are two subsets $W_{21}\subseteq W_2$ and $W_{22}\subseteq W_2$ with $|W_{21}|\geq |W_2|-3$ and $|W_{22}|\geq |W_2|-4$ so that for every distinct vertices $x'\in W_{21}$ and $y'\in W_{22}$, the path $\mathcal{Q}'_2=\Big((f_2\setminus\{x\})\cup\{x'\}\Big)\Big((g_2\setminus\{y\})\cup\{y'\}\Big)$ is also a good $\varpi_S$-configuration in $\mathcal{H}_{\rm blue}$ with  end vertices $x'$ and $y'$ in $W_2.$
 % Since each vertex of $W_2$ with the exception of at most $k-2,$ can be considered as an end vertex of $\mathcal{Q}_2,$
Therefore, we may assume that   $\bigcup_{i=1}^{2}\mathcal{Q}_i$ is either  a blue path or the union of  two disjoint blue
paths.
 Set $X_2=|W\setminus \bigcup_{i=1}^{2} V(\mathcal{Q}_i)|$ and
$\mathcal{P}_3=\mathcal{P}_{2}\setminus
\overline{\mathcal{P}}_{2}=e_7e_8\ldots e_n$. If
$\bigcup_{i=1}^{2}\mathcal{Q}_i$ is a blue path $\mathcal{Q}$  with end vertices $x_{2}$ and $y_{2}$,
then set
\begin{eqnarray*}
W_3=\Big(W_{2}\setminus
V(\mathcal{Q})\Big)\cup\{x_{2},y_{2}\}.
\end{eqnarray*}
In this case, clearly $|W_3|= |W_2|-1$.
Otherwise, $\bigcup_{i=1}^{2}\mathcal{Q}_i$ is the union of  two disjoint blue
paths $\mathcal{Q}$ and $\mathcal{Q}'$ with end vertices $x_{2},y_{2}$ and $x'_{2},y'_{2}$ in $W_2$, respectively. In this case,  set
\begin{eqnarray*}
W_3=\Big(W_{2}\setminus
V(\mathcal{Q}\cup\mathcal{Q}')\Big)\cup\{x_{2},y_{2},x'_{2},y'_{2}\}.
\end{eqnarray*}
Clearly $|W_3|=|W_2|$. If  $|W_3|\leq 5$ or $\|\mathcal{P}_3\|\leq 2$,
 then   $\bigcup_{i=1}^{2}\mathcal{Q}_i=\mathcal{Q}$ and  $\emptyset$  or $\mathcal{Q}$ and $\mathcal{Q}'$ (in the case $\bigcup_{i=1}^{2}\mathcal{Q}_i=\mathcal{Q}\cup\mathcal{Q}'$) are the paths between $W'=W\cap \bigcup_{i=1}^{2} V(\mathcal{Q}_i)$ and $\overline{\mathcal{P}}=\overline{\mathcal{P}}_1\cup \overline{\mathcal{P}}_2$
  with desired properties. Otherwise, go to Step $3$.\\\\\\

  %Either
  %$\mathcal{Q}=\mathcal{Q}_1\cup \mathcal{Q}_2$
%is  a blue path  with end vertices $x_2, y_2$
%in $W_2$ or we have  two disjoint blue paths $\mathcal{Q}=\mathcal{Q}_1$ and
%$\mathcal{Q}'=\mathcal{Q}_2$ with end vertices $x_2, y_2$ and $x'_2, y'_2$ in
%$W_2$.
%If the last edge of $\overline{\mathcal{P}}_2$ is $e_n$, then either $\mathcal{Q}_1\cup \mathcal{Q}_2$ and $\emptyset$ or $\mathcal{Q}_1$ and $\mathcal{Q}_2$ %are the paths with desired properties.\\\\
\noindent{\bf Step $\ell$ ($\ell>2$):}
Clearly  $|W_{\ell}|\geq 6$ and $\|\mathcal{P}_{\ell}\|\geq 3.$ Set
\begin{eqnarray*}
\hspace{-0.7 cm}&&\overline{\mathcal{P}}_{\ell}=e_{3{\ell}-2}e_{3{\ell}-1}e_{3{\ell}},\\
\hspace{-0.7 cm}&&\mathcal{P}'_{\ell}=\Big((e_{3{\ell}-2}\setminus\{f_{\mathcal{P},e_{3{\ell}-2}}\})\cup
\{u_{{\ell}-1}\}\Big)e_{3{\ell}-1}e_{3{\ell}}.
\end{eqnarray*}
    Since $\mathcal{P}$
   is maximal w.r.t. $W_{\ell}$,
 using Lemma \ref{spacial configuration2} there is a good $\varpi_S$-configuration, say
 $\mathcal{Q}_{\ell}=f_{\ell}g_{\ell},$ in $\mathcal{H}_{\rm blue}$ with end vertices $x\in f_{\ell}$ and $y\in g_{\ell}$ in $W_{\ell}$
% $\mathcal{P}'_l$ and $W'_l\subseteq W_l$
  such
 that   $\mathcal{Q}_{\ell}$ does not contain a vertex of  $e_{3{\ell}}\setminus e_{3{\ell}-1},$ say $u_{\ell}$.
 By Lemma \ref{spacial configuration2}, there are two subsets $W_{{\ell}1}\subseteq W_{\ell}$ and $W_{{\ell}2}\subseteq W_{\ell}$ with $|W_{{\ell}1}|\geq |W_{\ell}|-3$ and $|W_{{\ell}2}|\geq |W_{\ell}|-4$ so that for every distinct vertices $x'\in W_{{\ell}1}$ and $y'\in W_{{\ell}2}$, the path $\mathcal{Q}'_{\ell}=\Big((f_{\ell}\setminus\{x\})\cup\{x'\}\Big)\Big((g_{\ell}\setminus\{y\})\cup\{y'\}\Big)$ is also a good $\varpi_S$-configuration in $\mathcal{H}_{\rm blue}$ with  end vertices $x'$ and $y'$ in $W_{\ell}.$ Therefore, we may assume that
 either
  $\bigcup_{i=1}^{{\ell}}\mathcal{Q}_i$ is  a blue path $\mathcal{Q}$  with end
 vertices  in $W_{\ell}$  or we have  two disjoint blue paths $\mathcal{Q}$ and $\mathcal{Q}'$ with end vertices  in $W_{\ell}$   so that $\mathcal{Q}\cup \mathcal{Q}'=\bigcup_{i=1}^{{\ell}}\mathcal{Q}_i$.

\noindent Set $X_{\ell}=|W\setminus \bigcup_{i=1}^{{\ell}}V(\mathcal{Q}_i)|$ and
$\mathcal{P}_{{\ell}+1}=\mathcal{P}_{{\ell}}\setminus
\overline{\mathcal{P}}_{{\ell}}=e_{3{\ell}+1}e_{3{\ell}+2}\ldots e_n$. If
$\bigcup_{i=1}^{{\ell}}\mathcal{Q}_i$ is a blue path $\mathcal{Q}$ with end vertices $x_{{\ell}}$ and $y_{{\ell}}$,
then set
\begin{eqnarray*}
W_{{\ell}+1}=\Big(W_{{\ell}}\setminus
V(\mathcal{Q})\Big)\cup\{x_{{\ell}},y_{{\ell}}\}.
\end{eqnarray*}
Note that in this case, $|W_{{\ell}}|-2\leq |W_{{\ell}+1}|\leq |W_{{\ell}}|-1$.
Otherwise, $\bigcup_{i=1}^{{\ell}}\mathcal{Q}_i$ is the union of  two disjoint blue
paths $\mathcal{Q}$ and $\mathcal{Q}'$ with end vertices $x_{{\ell}},y_{{\ell}}$ and $x'_{{\ell}},y'_{{\ell}}$, respectively. In this case,  set
\begin{eqnarray*}
W_{{\ell}+1}=\Big(W_{{\ell}}\setminus
V(\mathcal{Q}\cup\mathcal{Q}')\Big)\cup\{x_{{\ell}},y_{{\ell}},x'_{{\ell}},y'_{{\ell}}\}.
\end{eqnarray*}
Clearly, $|W_{{\ell}}|-1\leq |W_{{\ell}+1}|\leq |W_{{\ell}}|$.\\
 If  $|W_{{\ell}+1}|\leq 5$ or $\|\mathcal{P}_{{\ell}+1}\|\leq 2$,
 then   $\bigcup_{i=1}^{{\ell}}\mathcal{Q}_i=\mathcal{Q}$ and  $\emptyset$  or $\mathcal{Q}$ and $\mathcal{Q}'$ (in the case $\bigcup_{i=1}^{{\ell}}\mathcal{Q}_i=\mathcal{Q}\cup\mathcal{Q}'$) are the paths with the desired properties. Otherwise, go to Step $\ell+1$.\\

Let $t\geq 2$ be the minimum integer for which we have either $|W_t|\leq
5$ or $\|\mathcal{P}_t\|\leq 2$.
 Set $x=X_{t-1}$ and $r=\|\mathcal{P}_{t}\|=n-3(t-1)$. So $\bigcup_{i=1}^{t-1}\mathcal{Q}_i$ is either  a blue path $\mathcal{Q}$ or the union two disjoint  blue paths $\mathcal{Q}$ and $\mathcal{Q}'$ between $\overline{\mathcal{P}}=e_1e_2\ldots
 e_{n-r}$ and $W'=W\cap (\bigcup_{i=1}^{t-1} V(\mathcal{Q}_i))$ with the desired properties.
 If $\bigcup _{i=1}^{t-1}\mathcal{Q}_i$ is  a blue path $\mathcal{Q},$ then either $x\in\{2,3\}$ or $x\geq
4$ and $0\leq r\leq 2$. Otherwise,
$\bigcup _{i=1}^{t-1}\mathcal{Q}_i$ is the union of  two disjoint blue paths
$\mathcal{Q}$ and $\mathcal{Q}'$ and we have either $x=1$ or $x\geq 2$ and $0\leq
r\leq 2$.

 }\end{proof}

\bigskip

%%%%%%%%%%%%%%%%%%%%%%%%%%%

%%%%%%%%%%%%%%%%%%%%%%%%%%%%

%%%%%%%%%%%%%%%%%%%%%%%%%%%%%%%%%%%%%%%%%%%%%%%%%%%%%%%%%%%%%%%%%%%%%%%%%%%%%%%%%%%%%%%%%%%%%%%%%%%%%%%%%%

%%%%%%%%%%%%%%%%%%
\section{Ramsey number of 5-uniform loose cycles}
In this section we determine the exact value of the
Ramsey number $R(\mathcal{C}^5_{n},\mathcal{C}^5_{m})$, where
$n\geq\lfloor\frac{3m}{2}\rfloor.$ We shall use Lemma \ref{there is P3:2} to
prove the following basic lemma.

\bigskip
\begin{lemma}\label{pm-1 implies pn-1}

Let $n= \Big\lfloor\frac{3m}{2}\Big\rfloor$, $m\geq 4$, and
$\mathcal{H}=\mathcal{K}^5_{4n+\lfloor\frac{m-1}{2}\rfloor}$ be
$2$-edge colored red and blue. If there is no copy of
$\mathcal{C}^5_{m}$ in $\mathcal{H}_{\rm blue}$ and
$\mathcal{C}=\mathcal{C}^5_{m-1}\subseteq\mathcal{H}_{\rm blue}$,
then $\mathcal{C}^5_{n-1}\subseteq\mathcal{H}_{\rm red}$.
\end{lemma}
%%%%%%%%%%%%%%%%%%%%%%%%%%%%
\begin{proof}{
Let
$\mathcal{C}=e_1e_2\ldots e_{m-1}$ be a copy of
$\mathcal{C}_{m-1}^5\subseteq \mathcal{H}_{\rm blue}$ with  edges
\begin{eqnarray*}e_j=\{v_{4j-3},v_{4j-2},v_{4j-1},v_{4j},v_{4j+1}\}\hspace{0.5 cm} (\rm{mod}\ \  4(m-1)),\hspace{0.5 cm} 1\leq j\leq m-1\end{eqnarray*} and
 $W=V(\mathcal{H})\setminus V(\mathcal{C})$. We have two following cases.

\bigskip
\noindent \textbf{Case 1.} For some edge $e_i=\{v_{4i-3},v_{4i-2},v_{4i-1},v_{4i},v_{4i+1}\},$ $1\leq i\leq m-1$, there are
two vertices $u,v\in W$ such that at least one of the edges
$\{v_{4i-1},v_{4i},v_{4i+1},u,v\}$ or
$\{v_{4i-3},v_{4i-2},v_{4i-1},u,v\}$ is blue.

\medskip

\noindent We can  assume that the edge
$e=\{v_{4i-1},v_{4i},v_{4i+1},u,v\}$ is blue.
 Set
\begin{eqnarray*}\mathcal{P}=e_{i+1}e_{i+2}\ldots e_{m-1}e_1e_2\ldots
e_{i-2}e_{i-1}\end{eqnarray*} and
%=f_1f_2\ldots f_{m-2} and
$W_0=W\setminus
\{u,v\}$ (if  the edge  $\{v_{4i-3},v_{4i-2},v_{4i-1},u,v\}$ is
blue,  consider the path
\begin{eqnarray*}
\mathcal{P}=e_{i-1}e_{i-2}\ldots
e_2e_1e_{m-1}\ldots e_{i+2}e_{i+1}
\end{eqnarray*}
 and do the
following process to get a red
copy of $\mathcal{C}_{n-1}^5$).\\

Let  $m=2k+p$, where $p=0,1$.
For $1\leq \ell \leq k-1,$ do  the following process.\\

\noindent{\bf Step 1:}
Set $\mathcal{P}_1=\mathcal{P}=g_1g_2\ldots g_{m-2}$, $W_1=W_0$ and
$\mathcal{P}'_1=\overline{\mathcal{P}}_1=g_1g_2$. Since $\mathcal{P}$ is maximal w.r.t.
$W_1$, using Lemma \ref{there is P3:2}, there is a red path $\mathcal{P}_3^5,$ say
$\mathcal{Q}_1,$ with  end vertices $x_1,y_1$ in $W_1$
 so that  $V(\mathcal{Q}_1)\subseteq V(\mathcal{P}'_1)\cup W_1$ and  $\mathcal{Q}_1$ does not contain a vertex   of $g_2\setminus g_1$,  say $u_1$. Let $W'_1=W_1\cap V(\mathcal{Q}_1).$ By Lemma  \ref{there is P3:2}, we have  $|W'_1|\leq 6$.\\\\
%  and $W'_1\subseteq
%W_1$, $|W'_1|\leq 6$, so that
{\bf Step $\ell$ ($ 2\leq \ell \leq k-1$):} Set
$\mathcal{P}_{\ell}=\mathcal{P}_{\ell-1}\setminus \overline{\mathcal{P}}_{\ell-1}=g_{2\ell-1}g_{2\ell}\ldots g_{m-2}$,
$\mathcal{P}'_{\ell}=\Big((g_{2\ell-1}\setminus
\{f_{\mathcal{P},g_{2\ell-1}}\})\cup\{u_{\ell-1}\}\Big)g_{2\ell}$,
$\overline{\mathcal{P}}_{\ell}=g_{2\ell-1}g_{2\ell}$ and $W_{\ell}=(W_{1}\setminus
\bigcup_{j=1}^{\ell-1}V(\mathcal{Q}_j))\cup\{x_{\ell-1},y_{\ell-1}\}$. Since $\mathcal{P}$ is maximal w.r.t. $W_{\ell},$ using
% Since $\mathcal{P}$ is maximal w.r.t. $W_2$,
 Lemma \ref{there is P3:2}, there is  a  red path $\mathcal{P}_3^5,$ say $\mathcal{Q}_{\ell},$  with   mentioned  properties  so that $V(\mathcal{Q}_{\ell})\subseteq V(\mathcal{P}'_{\ell})\cup W_{\ell}$ and $\mathcal{Q}_{\ell}$ does not contain a vertex of $g_{2\ell}\setminus g_{2\ell-1},$ say $u_{\ell}$. Since each vertex of $W_{\ell},$ with the exception of at most one, can be considered as an end vertex of $\mathcal{Q}_{\ell},$ we may assume that
  $\bigcup_{j=1}^{\ell}\mathcal{Q}_j$
is  a red path with end vertices $x_{\ell}, y_{\ell}$ in $W_{\ell}$.
  Let $W'_{\ell}=W_{1}\cap V(\bigcup_{j=1}^{\ell} \mathcal{Q}_j)$. Clearly $|W'_{\ell}|\leq 5\ell+1$.\\
% and  $\mathcal{Q}_1\cup \mathcal{Q}_2$
%is  a red path with end vertices $x_2, y_2$ in $W_2$.

  Therefore, $\mathcal{Q}=\bigcup_{\ell=1}^{k-1} \mathcal{Q}_{\ell}$ is a red path of length $3k-3$ with end vertices $x_{k-1},
y_{k-1}$ in $W_{k-1}$ and $|W'_{k-1}|=|W_0\cap V(\mathcal{Q})|\leq 5(k-1)+1$. Let $W'=W_0\setminus V(\mathcal{Q})=W_0\setminus W'_{k-1}$.
% and $|W'|=5+i,$ $i=0,1,$
% where $W'=W_0\setminus V(\mathcal{Q})$.
 We have two following  subcases.\\

 \medskip
\noindent {\it Subcase $1$}. $m=2k.$

\medskip
\noindent Since $|W_0|=5k+1,$ we have   $|W'|\geq 5$. Let $\{u_1,u_2,\ldots,u_5\}\subseteq W'$ and $z$ be a vertex
of
$g_{m-2}\setminus g_{m-3}$
 ($e_{i-1}\setminus e_{i-2}$) so that $z\notin V(\mathcal{Q})$. Since there is no blue  copy of
$\mathcal{C}^5_{m}$ and the edge $e$ is blue, the edges
%\begin{eqnarray*}
$$f_1=\{y_{k-1},u_1,u_2,v_{4i-1},z\}, f_2=\{z,v_{4i-2},u,u_3,x_{k-1}\},$$
%\end{eqnarray*}
 are red and $\mathcal{Q}f_1f_2$
 is a red copy of $\mathcal{C}_{n-1}^5$ (note that $n=3k$).

\medskip
\noindent{\it Subcase $2$}. $m=2k+1.$

\medskip
\noindent In this case we have $n=3k+1$ and  $|W'|\geq 6.$ Let $\{u_1,$ $u_2,\ldots,u_6\}\subseteq W'$.  Also we have    $(e_{i-1}\setminus\{f_{\mathcal{P},e_{i-1}}\})\cap V(\mathcal{Q})=\emptyset.$  Since there is no blue copy of
$\mathcal{C}^5_{m}$ and the edge $e$ is blue, the edges
% \begin{eqnarray*}
$$g_1=\{y_{k-1},v_{4i-4},v_{4i-1},u_1,u_2\}, g_2=\{u_2,u_3,
u_4,v_{4i-5},v_{4i}\},$$ $$g_3=\{v_{4i},v_{4i-3},u_5,u_6,x_{k-1}\},$$
%\end{eqnarray*}
are red and $\mathcal{Q}g_1g_2g_3$
 is
a copy of  $\mathcal{C}_{n-1}^5$ in $\mathcal{H}_{\rm red}$.

%\bigskip
%\noindent \textbf{Case 2. } There are an edge
%$e_i=\{v_{4i-3},v_{4i-2},v_{4i-1},v_{4i},v_{4i+1}\}$, $1\leq i\leq
%m-1$, and two vertices $u,v\in W$ such that
%$\{v_{4i-3},v_{4i-2},v_{4i-1},u,v\}$ is blue.

%\medskip
%In this case, consider the path $
%\mathcal{P}_1=e_{i-1}e_{i-2}\ldots e_2e_1e_{m-1}e_{m-2}\ldots
%e_{i+2}e_{i+1}$ and repeat the proof of Case $1$. By
%  an argument similar  we can find a red copy of
%  $\mathcal{C}^5_{n-1}$ and so we are done.

\bigskip
\noindent \textbf{Case 2.} For every edge
$e_i=\{v_{4i-3},v_{4i-2},v_{4i-1},v_{4i},v_{4i+1}\}$, $1\leq i\leq
m-1$, and every vertices $u,v\in W$  the edges
$\{v_{4i-1},v_{4i},v_{4i+1},u,v\}$ and
$\{v_{4i-3},v_{4i-2},v_{4i-1},u,v\}$ are red.

\medskip
\noindent Clearly
\begin{eqnarray*}
 |W|=\left\lbrace \begin{array}{ll}
5k+3
&\mbox{if $m=2k$},\vspace{.5 cm}\\
5k+4 &\mbox{if $m=2k+1$}.
\end{array}\right. \end{eqnarray*}
\noindent  Now let $W'=\{x_1,x_2,\ldots,x_{5k+3}\}\subseteq W$.
For   $0\leq j\leq k-2$ and $1\leq \ell \leq 3$,  set
\begin{eqnarray*}
 f_{3j+\ell}=\left\lbrace \begin{array}{ll}
(e_{2j+1}\setminus
\{v_{8j+4},v_{8j+5}\})\cup\{x_{5j+1},x_{5j+2}\}
&\mbox{if $\ell=1$},\vspace{.5 cm}\\
(e_{2j+1}\setminus
\{v_{8j+1},v_{8j+2}\})\cup\{x_{5j+3},x_{5j+4}\} &\mbox{if $\ell=2$},\vspace{.5 cm}\\
(e_{2j+2}\setminus
\{v_{8j+5},v_{8j+6}\})\cup\{x_{5j+4},x_{5j+5}\} &\mbox{if $\ell=3$},
\end{array}\right. \end{eqnarray*}
where $$e_{2j+1}=\{v_{8j+1},v_{8j+2},v_{8j+3},v_{8j+4},v_{8j+5}\}, e_{2j+2}=\{v_{8j+5},v_{8j+6},v_{8j+7},v_{8j+8},v_{8j+9}\}.$$
Also, let
\begin{eqnarray*}
 f_{3(k-1)+\ell}=\left\lbrace \begin{array}{ll}
(e_{2k-1}\setminus
\{v_{8k-4},v_{8k-3}\})\cup\{x_{5k-4},x_{5k-3}\}
&\mbox{if $\ell=1$},\vspace{.5 cm}\\
(e_{2k-1}\setminus
\{v_{8k-7},v_{8k-6}\})\cup\{x_{5k-2},x_{5k-1}\} &\mbox{if $\ell=2$},\vspace{.5 cm}\\
(e_{2k}\setminus
\{v_{8k-3},v_{8k-2}\})\cup\{x_{5k-1},x_{5k}\} &\mbox{if $\ell=3$ and $m=2k+1$}.\vspace{.5 cm}
\end{array}\right. \end{eqnarray*}
 Clearly for $m=2k,$ $\mathcal{C}'
 =f_1f_2 \ldots f_{3k-1}$
% $\mathcal{C}=e_1^{\prime}e_1^{\prime\prime}e_2^{\prime}e_3^{\prime}e_3^{\prime\prime}e_4^{\prime}e_5^{\prime}\ldots
%e_{2k-2}^{\prime} e_{2k-1}^{\prime}e_{2k-1}^{\prime\prime}$ and
 and for $m=2k+1$, $\mathcal{C}''=f_1f_2 \ldots f_{3k}$
%$\mathcal{C}=e_1^{\prime}e_1^{\prime\prime}e_2^{\prime}e_3^{\prime}e_3^{\prime\prime}e_4^{\prime}e_5^{\prime}\ldots
%e_{2k-2}^{\prime} e_{2k-1}^{\prime}e_{2k-1}^{\prime\prime}e_{2k}^{\prime}$
is   a copy of
 $\mathcal{C}_{n-1}^5$ in $\mathcal{H}_{\rm red}$.
}\end{proof}

%%%%%%%%%%%%%%%%%%%%%%%%%%%%
\begin{lemma}\label{Cn-1 implies small Cm}
Let $n\geq \Big\lfloor\frac{3m}{2}\Big\rfloor$, $6\geq m\geq 4$ and
$\mathcal{H}=\mathcal{K}^5_{4n+\lfloor\frac{m-1}{2}\rfloor}$ be
$2$-edge colored red and blue. If there is no red copy of
$\mathcal{C}^5_{n}$  and $\mathcal{C}=\mathcal{C}^5_{n-1}\subseteq
\mathcal{H}_{\rm red}$, then $\mathcal{C}^5_{m}\subseteq
\mathcal{H}_{\rm blue}$.
\end{lemma}

\begin{proof}{Let
$\mathcal{C}=e_1e_2\ldots e_{n-1}$ be a copy of
$\mathcal{C}_{n-1}^5$ in $\mathcal{H}_{\rm red}$ with  edges
\begin{eqnarray*}e_i=\{v_{4i-3},v_{4i-2},v_{4i-1},v_{4i},v_{4i+1}\}\hspace{0.5 cm} (\rm{mod}\ \  4(n-1)),\hspace{0.5 cm} 1\leq i\leq n-1
\end{eqnarray*} and  $W=V(\mathcal{H})\setminus V(\mathcal{C})$.
We
have two following cases.

\bigskip
\noindent \textbf{Case 1.} For some edge $e_i=\{v_{4i-3},v_{4i-2},v_{4i-1},v_{4i},v_{4i+1}\}$, $1\leq i\leq n-1$,  there are
 vertices $z_1,z_2\in W$  such that at least one of the edges
$\{v_{4i-1},v_{4i},v_{4i+1},z_1,z_2\}$ or
$\{v_{4i-3},v_{4i-2},v_{4i-1},z_1,z_2\}$ is red.

\medskip
\noindent We can suppose  that there are
 vertices $z_1,z_2\in W$  so that the edge
$e=\{v_{4i-1},v_{4i},v_{4i+1},z_1,z_2\}$ is red. Set
\begin{eqnarray*}\mathcal{P}=e_{i+1}e_{i+2}\ldots e_{n-1}e_1e_2\ldots
e_{i-2}e_{i-1}\end{eqnarray*}
 and  $W_0=W\setminus
\{z_1,z_2\}$. If  the edge $\{v_{4i-3},v_{4i-2},v_{4i-1},z_1,z_2\}$ is
red, then  consider the path \begin{eqnarray*}\mathcal{P}=e_{i-1}e_{i-2}\ldots
e_2e_1e_{n-1}\ldots e_{i+2}e_{i+1}\end{eqnarray*} and do the
following process   to get a blue copy of $\mathcal{C}_{m}^5$.\\

%It is not difficult to check that if $m\leq 6,$ there is a blue copy of $\mathcal{C}^5_m$
%(similar to Lemma \ref{cn-1 implies cm for n>m}). So we may assume
First let $m=4$. Hence  we have $n\geq 6$ and $|W_0|=3$. Let $W_0=\{u_1,u_2,u_3\}$.
 Since there is no red copy
of $\mathcal{C}^5_n$, $\mathcal{P}$ is a maximal path w.r.t.
$W$. Use Lemma \ref{spacial configuration2} (by putting $u=f_{\mathcal{P},e_{i-4}}$) to obtain  a good $\varpi_{S}$-configuration, say $C_1$, in $\mathcal{H}_{\rm blue}$ with end vertices in $W$ and $S\subseteq e_{i-4}e_{i-3}e_{i-2}$. Since $C_1$ is a good $\varpi_{S}$-configuration, there is a vertex $w\in e_{i-2}\setminus e_{i-3}$ so that $w\notin V(C_1).$  We have two following cases. \\

\begin{itemize}
\item[(i)] $V(C_1)\cap W_0\neq \emptyset$. \\
Assume that $u_1\in V(C_1)\cap W_0$. Since $|W_0|=3$,   we have $W_0\setminus V(C_1)\neq \emptyset$. Suppose that  $u_2\in W_0\setminus V(C_1)$.
Let $f=\{u_1,w,v_{4i-6},v_{4i-5},u_2\}.$ If the edge  $f$ is blue, then set
\begin{eqnarray*}
&&g_1=\{u_2,v_{4i-3},v_{4i-2},u_3,z_1\},\\
&&g_2=\{u_2,v_{4i-3},v_{4i-2},u_3,z_2\}.
\end{eqnarray*}
Since there is no red copy of $\mathcal{C}_{n}^5$ and the edge $e$ is red, then the edges $g_1$ and $g_2$ are blue.
%\begin{eqnarray*}
%&&\mathcal{C}'=C_1f\{u_2,u_3,z_1,v_{4i-3},v_{4i-2}\},\\
%&&\mathcal{C}''=C_1f\{u_2,u_3,z_2,v_{4i-3},v_{4i-2}\}.
%\end{eqnarray*}
Clearly at least one of $\mathcal{C}'=C_1fg_1$ or  $\mathcal{C}''=C_1fg_2$ is a blue copy of $\mathcal{C}_4^5$. So we may assume that the edge $f$ is red. If $\{z_1,z_2\}\cap V(C_1)\neq \emptyset$, then set $g=\{u_2,v_{4i-4},v_{4i-3},z_1,z_2\}.$ The edge $g$ is blue (otherwise, $fge_ie_{i+1}\ldots e_{n-1}e_1\ldots e_{i-2}$ is a red copy of $\mathcal{C}^5_{n},$ a contradiction). Again, since there is no red copy of $\mathcal{C}_{n}^5$ and the edge $e$ is red,
  $C_1g \{v_{4i-3},v_{4i-2},v_{4i-1},u_3,u_1\}$ is a blue copy of $\mathcal{C}_4^5$. So we may suppose  that  $\{z_1,z_2\}\cap V(C_1)= \emptyset.$ Therefore, $u_1$ and $u_3$ are end vertices of $C_1$ in $W.$ Clearly
\begin{eqnarray*}
C_1\{u_1,z_1,z_2,v_{4i-4},v_{4i-3}\}\{v_{4i-3},v_{4i-2},v_{4i-1},u_2,u_3\}
\end{eqnarray*}
 is a blue copy of $\mathcal{C}_4^5$.

\item[(ii)] $V(C_1)\cap W_0= \emptyset$.\\
Therefore $z_1$ and $z_2$ are end vertices of $C_1$ in $W$. Let $f=\{z_1,w,v_{4i-6},v_{4i-5},u_1\}.$ If the edge  $f$ is blue, then
\begin{eqnarray*}
C_1f\{u_1,u_2,v_{4i-3},v_{4i-2},z_2\},
\end{eqnarray*}
  is a blue copy of $\mathcal{C}_4^5$. Otherwise, since there is no red copy of $\mathcal{C}_{n}^5,$ the edge $g=\{z_1,u_2,u_3,v_{4i-4},v_{4i-3}\}$ is blue and $C_1g\{u_2,u_1,v_{4i-5},v_{4i-2},z_2\}$ is a blue copy of $\mathcal{C}_4^5$.
\end{itemize}

Now let $m=5$. Therefore, we have $n\geq 7$ and $|W_0|=4$. Let $W_0=\{u_1,\ldots,u_4\}$.
 Since there is no red copy
of $\mathcal{C}^5_n$, $\mathcal{P}$ is a maximal path w.r.t.
$W'=W_0\cup \{z_1\}$. Use Lemma \ref{spacial configuration2} (by putting $u=f_{\mathcal{P},e_{i-5}}$) to obtain a good $\varpi_{S}$-configuration, say $C_2$, in $\mathcal{H}_{\rm blue}$ with end vertices in $W'$ and $S\subseteq e_{i-5}e_{i-4}e_{i-3}$. Since $C_2$ is a good $\varpi_{S}$-configuration, there is a vertex $w\in e_{i-3}\setminus e_{i-4}$ so that $w\notin V(C_2).$
%By Lemma \ref{spacial configuration1} each vertex of $W'$ with the exception of at most $3,$ can be considered as an end vertex of $C_2.$
Clearly
 $V(C_2)\cap W_0\neq \emptyset$. By symmetry we may assume that $u_1$ is an end vertex of $C_2$ in $W_0$. Since $|V(C_2)\cap W'|=2$, we have $|W_0\setminus V(C_2)|\geq 2$. Without loss of generality suppose  that $u_2,u_3\in W_0\setminus V(C_2)$. Let $f'=\{u_1,v_{4i-9},v_{4i-10},w,u_2\}$.
  %where $x\in (e_{i-3}\setminus e_{i-4})\setminus V(C_2)$.
  If the edge $f'$ is blue, then set
\begin{eqnarray*}
&&\mathcal{C}'=C_2f'\{u_2,v_{4i-7},v_{4i-6},v_{4i-5},u_3\}\{u_3,u_4,z_1,v_{4i-3},v_{4i-2}\},\\
&&\mathcal{C}''=C_2f'\{u_2,u_3,v_{4i-1},v_{4i-2},v_{4i-3}\}\{v_{4i-3},v_{4i-4},v_{4i-5},z_1,u_4\}.
\end{eqnarray*}
Clearly at least one of $\mathcal{C}'$ or $\mathcal{C}''$ is a blue copy of $\mathcal{C}_5^5$. So we may assume that the edge $f'$ is red. Since there is no red copy of $\mathcal{C}_n^5,$ the edge $g'=\{u_1,u_3,z_2,v_{4i-7},v_{4i-8}\}$ is blue. If the edge $g''=\{u_3,v_{4i-9},v_{4i-6},v_{4i-5},u_2\}$ is blue, then $C_2g'g''\{u_2,v_{4i-3},v_{4i-2},z_1,u_4\}$ is a blue copy of $\mathcal{C}_5^5$. Otherwise,
\begin{eqnarray*}
C_2g'\{u_3,u_2,v_{4i-1},v_{4i-2},v_{4i-3}\}\{v_{4i-3},v_{4i-4},v_{4i-5},z_1,u_4\}
\end{eqnarray*}
 is a blue copy of $\mathcal{C}_5^5$.\\\\

Now let $m=6$. So we have $n\geq 9$ and $|W_0|=4$.   Since $\mathcal{P}$ is
maximal w.r.t. $W$, using Lemma
 \ref{spacial configuration2} there is a good  $\varpi_{S}$-configuration, say $\mathcal{Q}_1$, in $\mathcal{H}_{\rm blue}$ with end vertices $x_1$ and $y_1$ in $W$ so that
 $S\subseteq e_{i-7}\cup e_{i-6}\cup e_{i-5}$  and $\mathcal{Q}_1$ does not contain  a vertex
 of $e_{i-5}\setminus e_{i-6},$ say $u$. Now, Since $\mathcal{P}$ is
maximal w.r.t. $W\setminus\{x_1\}$,  using Lemma \ref{spacial configuration2} there is a good $\varpi_{S}$-configuration, say $\mathcal{Q}_2$, in $\mathcal{H}_{\rm blue}$ with end vertices  in $W\setminus \{x_1\}$ so that
 $S\subseteq \Big((e_{i-4}\setminus\{f_{\mathcal{P},e_{i-4}}\})\cup\{u\}\Big)\cup e_{i-3}\cup e_{i-2}$  and $\mathcal{Q}_2$ does not contain
 a vertex of  $e_{i-2}\setminus e_{i-3},$ say $u'$. Clearly, $\mathcal{Q}_1\cup \mathcal{Q}_2$ is either a blue path of length $4$ or the union of two disjoint blue paths of length $2$. First consider the case $\mathcal{Q}_1\cup \mathcal{Q}_2$ is a blue path, say $\mathcal{Q}$, of length $4$ with end vertices $x_1$ and $y_2$ in $W$. It is easy to see that one of the following cases holds.

\begin{itemize}
\item[(i)]  $|\{x_1,y_2\}\cap \{z_1,z_2\}|=0$\\
In this case, obviously $|V(\mathcal{Q})\cap \{z_1,z_2\}|\leq 1$.
By symmetry we may suppose  that $z_2\notin V(\mathcal{Q})$. Since $|W_0|=4$, we have $|W_0\setminus V(\mathcal{Q})|\geq 1$. Let $u_1\in W_0\setminus V(\mathcal{Q})$. Consider the edge $f=\{y_2,u',v_{4i-6},v_{4i-5},u_1\}$. If the edge $f$ is blue, since there is no red copy of $\mathcal{C}_n^5$ and the edge $e$ is red,
  then $\mathcal{Q}f\{u_1,v_{4i-3},v_{4i-2},z_2,x_1\}$ is a blue copy of $\mathcal{C}_6^5$. Otherwise,
set $T=W\setminus (V(\mathcal{Q})\cup \{u_1\})$. So we have $z_2\in T$ and $|T|=2$. The edge $f'=\{y_2,v_{4i-4},v_{4i-3}\}\cup T$ is blue (otherwise, $ff'e_ie_{i+1}\ldots e_{n-1}e_1\ldots e_{i-2}$ is a red copy of $\mathcal{C}^5_{n},$ a contradiction). Thereby, $\mathcal{Q}f' \{v_{4i-3},v_{4i-2},v_{4i-1},u_1,x_1\}$ is a blue copy of $\mathcal{C}_6^5$.

\item[(ii)]  $|\{x_1,y_2\}\cap \{z_1,z_2\}|=1$\\
 By symmetry we may assume that $y_2=z_1$. Since $|W_0|=4$ and $|W\cap V(\mathcal{Q})|=3$, we have $|W_0\setminus V(\mathcal{Q})|\geq 2$. Let $u_1\in W_0\setminus V(\mathcal{Q})$. Consider the edge $f=\{y_2,u',v_{4i-6},v_{4i-5},u_1\}$. If the edge $f$ is blue, then $\mathcal{Q}f\{u_1,v_{4i-3},v_{4i-2},v_{4i-1},x_1\}$ is a blue copy of $\mathcal{C}_6^5$. Otherwise,
set $T=W\setminus (V(\mathcal{Q})\cup \{u_1\})$.  Clearly, the edge $f'=\{y_2,v_{4i-4},v_{4i-3}\}\cup T$ is blue and thereby, $\mathcal{Q}f' \{v_{4i-3},v_{4i-2},v_{4i-1},u_1,x_1\}$ is a blue copy of $\mathcal{C}_6^5$.

\item[(iii)]  $|\{x_1,y_2\}\cap \{z_1,z_2\}|=2$\\
In this case, clearly $|W_0\setminus V(\mathcal{Q})|=3$. Set $T=W_0\setminus V(\mathcal{Q})=\{u_1,u_2,u_3\}$. If the edge $f=\{y_2,u',v_{4i-6},v_{4i-5},u_1\}$ is blue, then $\mathcal{Q}f\{u_1,v_{4i-3},v_{4i-2},u_2,x_1\}$ is a blue copy of $\mathcal{C}_6^5$. Otherwise, the edge $f'=\{y_2,v_{4i-4},v_{4i-3},u_2,u_3\}$ is blue and  $\mathcal{Q}f' \{u_3,u_1,v_{4i-5},v_{4i-2},x_1\}$ is a blue copy of $\mathcal{C}_6^5$.
\end{itemize}

Now let $\mathcal{Q}_1$ and $\mathcal{Q}_2$ are two disjoint blue paths with end vertices $x_1,y_1$ and $x_2,y_2$ in $W,$ respectively. We have the following cases.

\begin{itemize}
\item[(i)] $|\{x_1,y_1,x_2,y_2\}\cap \{z_1,z_2\}|=0$\\
 Let $f=\{y_1,u',v_{4i-6},v_{4i-5},x_2\}$. If the edge $f$ is blue, since there is no red copy of $\mathcal{C}_n^5$ and the edge $e$ is red,  then
 \begin{eqnarray*}
 \mathcal{Q}_1f\mathcal{Q}_2\{y_2,v_{4i-3},v_{4i-2},v_{4i-1},x_1\},
  \end{eqnarray*}
  is a blue copy of $\mathcal{C}_6^5$. Otherwise, the edge $f'=\{x_1,v_{4i-5},v_{4i-4},v_{4i-3},y_2\}$ is blue (if not, we can find a red copy of $\mathcal{C}_n^5,$ a contradiction). Thereby, $\mathcal{Q}_1f'\mathcal{Q}_2\{x_2,v_{4i-6},v_{4i-2},v_{4i-1},y_1\}$
 is a blue copy of $\mathcal{C}_6^5$.

 \item[(ii)] $|\{x_1,y_1,x_2,y_2\}\cap \{z_1,z_2\}|\geq 1$\\
 By symmetry we may assume that $y_2=z_1.$
 Clearly $|W_0\setminus V(\mathcal{Q}_1\cup \mathcal{Q}_2)|,|W_0 \cap V(\mathcal{Q}_1)|\geq 1.$ We can  without loss of generality suppose that $u_1\in W_0\setminus V(\mathcal{Q}_1\cup\mathcal{Q}_2) $ and $x_1\in W_0 \cap V(\mathcal{Q}_1).$   If the edge $f=\{y_1,u',v_{4i-6},v_{4i-5},y_2\}$ is blue, then set
 \begin{eqnarray*}
 &&g=\{x_2,v_{4i-3},v_{4i-2},u_1,x_1\},\\
 &&g'=\{x_2,v_{4i-3},v_{4i-2},v_{4i-1},x_1\}.
 \end{eqnarray*}
 If $x_2=z_2,$ then $\mathcal{Q}_1f\mathcal{Q}_2g $  is a blue copy of $\mathcal{C}_6^5$ (if not, $gee_{i+1}\ldots e_{n-1}e_1\ldots e_{i-1}$ is a red copy of $\mathcal{C}_n^5,$ a contradiction). Otherwise, the edge $g'$ is blue (if not, $g'ee_{i+1}\ldots e_{n-1}e_1\ldots e_{i-1}$ is a red copy of $\mathcal{C}_n^5$) and
  $\mathcal{Q}_1f\mathcal{Q}_2g'$ is a blue copy of $\mathcal{C}_6^5$.\\
  Therefore, we may assume that the edge $f$ is red. Since there is no red copy of $\mathcal{C}_n^5,$ the edges $h_1=\{x_1,v_{4i-5},v_{4i-4},v_{4i-3},x_2\}$ and $h_2=\{y_1,v_{4i-4},v_{4i},u_1,x_2\}$ are blue (if the edge $h_j,$ $1\leq j \leq 2,$  is red, then $fh_je_{i}\ldots e_{n-1}e_1\ldots e_{i-2}$ is a red copy of $\mathcal{C}_n^5$).
 If $y_1\neq z_2,$ then $\mathcal{Q}_1h_1\mathcal{Q}_2 \{y_2,u_1,v_{4i-2},v_{4i-6},y_1\}$ is a blue copy of $\mathcal{C}_6^5$. If $y_1=z_2,$ then there is a vertex $u_2\in W_0\setminus(V(\mathcal{Q}_1\cup \mathcal{Q}_2)\cup\{u_1\})$ and clearly $\mathcal{Q}_1h_2\mathcal{Q}_2 \{y_2,v_{4i-3},v_{4i-2},u_2,x_1\}$ is a blue copy of $\mathcal{C}_6^5$.

\end{itemize}

\bigskip
\noindent \textbf{Case 2. } For every edge
$e_i=\{v_{4i-3},v_{4i-2},v_{4i-1},v_{4i},v_{4i+1}\}$, $1\leq i\leq
n-1$, and every vertices $z_1,z_2\in W$  the edges $\{v_{4i-1},v_{4i},v_{4i+1},z_1,z_2\}$
and $\{v_{4i-3},v_{4i-2},v_{4i-1},z_1,z_2\}$ are blue.

\medskip
\noindent Let $W=\{u_1,u_2,\ldots,u_{\lfloor\frac{m-1}{2}\rfloor+4}\}$.
Set
\begin{eqnarray*}
h_i=\left\lbrace \begin{array}{ll}
(e_i\setminus \{v_{4i},v_{4i+1}\})\cup \{u_i,u_{i+1}\}
&\mbox{for $1\leq i \leq m-1$},\vspace{.5 cm}\\
(e_i\setminus \{v_{4i},v_{4i+1}\})\cup \{u_i,u_{1}\} &\mbox{for $i=m$}.
\end{array}\right.
 \end{eqnarray*}
Since $4\leq m \leq 6$ and  $h_i$'s, $1\leq i \leq m$, are blue, then   $h_1h_2\ldots h_m$ is a blue $\mathcal{C}^5_m$.

}\end{proof}
%%%%%%%%%%%%%%%%%%%%%%%%%%%%

\begin{lemma}\label{Cn-1 implies cm}
Let $n\geq \Big\lfloor\frac{3m}{2}\Big\rfloor$, $m\geq 7$ and
$\mathcal{H}=\mathcal{K}^5_{4n+\lfloor\frac{m-1}{2}\rfloor}$ be
$2$-edge colored red and blue. If there is no red copy of
$\mathcal{C}^5_{n}$  and $\mathcal{C}=\mathcal{C}^5_{n-1}\subseteq
\mathcal{H}_{\rm red}$, then $\mathcal{C}^5_{m}\subseteq
\mathcal{H}_{\rm blue}$.
\end{lemma}

%%%%%%%%%%%%%%%%%%%%%%%%%%%%
\begin{proof}{Let
$\mathcal{C}=e_1e_2\ldots e_{n-1}$ be a copy of
$\mathcal{C}_{n-1}^5$ in $\mathcal{H}_{\rm red}$ with  edges
\begin{eqnarray*}e_i=\{v_{4i-3},v_{4i-2},v_{4i-1},v_{4i},v_{4i+1}\}\hspace{0.5 cm} (\rm mod\ \  4(n-1)),\hspace{0.5 cm} 1\leq i\leq n-1.
\end{eqnarray*} Also, let $W=V(\mathcal{H})\setminus V(\mathcal{C})$. We
have the following cases:

\bigskip
\noindent \textbf{Case 1. } For some edge $e_i=\{v_{4i-3},v_{4i-2},v_{4i-1},v_{4i},v_{4i+1}\}$, $1\leq i\leq n-1$,  there are
 vertices $z_1,z_2\in W$ and $v'\in e_{i+1}\setminus\{l_{\mathcal{C},e_{i+1}}\}$ (resp. $z_1,z_2\in W$ and $v'\in e_{i-1}\setminus\{f_{\mathcal{C},e_{i-1}}\}$) such that the edge
$\{v_{4i-1},v_{4i},v',z_1,z_2\}$ (resp. the edge
$\{v',v_{4i-2},v_{4i-1},z_1,z_2\}$) is red.

\medskip
 We can without loss of generality assume that there are
 vertices $z_1,z_2\in W$ and $v'\in e_{i+1}\setminus\{l_{\mathcal{C},e_{i+1}}\}$ so that the edge
$e=\{v_{4i-1},v_{4i},v',z_1,z_2\}$ is red. Set
\begin{eqnarray*}\mathcal{P}=e_{i+1}e_{i+2}\ldots e_{n-1}e_1e_2\ldots
e_{i-2}e_{i-1}\end{eqnarray*}
 and  $W_0=W\setminus
\{z_1,z_2\}$. If there are
 vertices $z_1,z_2\in W$ and $v'\in e_{i-1}\setminus\{f_{\mathcal{C},e_{i-1}}\}$ so that the edge $\{v',v_{4i-2},v_{4i-1},z_1,z_2\}$ is
red,  consider the path \begin{eqnarray*}\mathcal{P}=e_{i-1}e_{i-2}\ldots
e_1e_{n-1}e_{n-2}\ldots e_{i+2}e_{i+1}\end{eqnarray*} and repeat the
following process  to get a blue copy of $\mathcal{C}_{m}^5$.\\

%It is not difficult to check that if $m\leq 6,$ there is a blue copy of $\mathcal{C}^5_m$
%(similar to Lemma \ref{cn-1 implies cm for n>m}). So we may assume
Since $m\geq 7$,  we have  $|W_0|=\lfloor\frac{m-1}{2}\rfloor+2 \geq 5$. Also,  since there is no red copy
of $\mathcal{C}^5_n$, $\mathcal{P}$ is a maximal path w.r.t.
$W_0$. Applying Lemma \ref{there is a Pl}, there are two
disjoint blue paths
 between
 $\overline{\mathcal{P}}$, the path obtained from $\mathcal{P}$ by deleting the last $r$ edges for
 some $r\geq 0,$ and $W'\subseteq W_0$ with the mentioned properties. Consider the paths
 $\mathcal{Q}$ and $\mathcal{Q}'$ with $\|\mathcal{Q}\|\geq \|\mathcal{Q}'\|$
  so that $\ell'=\|\mathcal{Q}\cup\mathcal{Q}'\|$ is maximum. Among these paths,  choose  $\mathcal{Q}$ and $\mathcal{Q}'$, where $\|\mathcal{Q}\|$ is maximum. Since $\|\mathcal{P}\|=n-2,$  by Lemma \ref{there is a Pl}, we have  $\ell'=\frac{2}{3}(n-2-r)$.\\
  %Without loss of generality we may assume that $\|\mathcal{Q}\|\geq \|\mathcal{Q}'\|.$
  %Among these paths choose $\mathcal{Q}$ and $\mathcal{Q}'$, where $\|\mathcal{Q}\|$ is maximum.
  %By Lemma \ref{there is a Pl},

\noindent {\it Subcase $1$.} $\|\mathcal{Q}'\|\neq 0$.\\
 %Now we may assume that  $\|\mathcal{Q}'\|\neq 0$.
Set $T=W_0\setminus W'.$ Let  $x,y$ and $x',y'$ be the end
vertices of $\mathcal{Q}$ and $\mathcal{Q}'$ in $W',$
respectively. Using Lemma \ref{there is a Pl} we have one of
the following cases:

%Now we may assume that there is two blue path $\mathcal{Q}'$ and
%$\mathcal{Q}''$ between $\overline{\mathcal{P}}$ and  $W'$. Let
%$x',y'$ and $x'',y''$ be end vertices of $\mathcal{Q}'$ and
%$\mathcal{Q}''$ in $W',$ respectively,
% $T=W_0\setminus V(\mathcal{Q})$ and $b=|T|$. Using Lemma \ref{there is a Pl} we have one of  the following
% subcases.\\

\medskip
%{\it Subcase $1$}. $b=0.$
\begin{itemize}
\item[I.] $|T|\geq 2.$

\medskip
In this case, we have $\ell'\leq 2\lfloor\frac{m-1}{2}\rfloor-4$ and so
 $r\geq 4$. Therefore,  this case
does not occur by Lemma \ref{there is a Pl}.\\

%\begin{itemize}
\item[II.] $|T|=1.$

\medskip
Let $T=\{u_1\}$. Clearly $\ell'=2\lfloor\frac{m-1}{2}\rfloor-2$. If
 $m$ is odd, then $\ell'=m-3$ and $r\geq 2$. Set $h=\{y,w,v_{4i-10},v_{4i-9},x'\}$,  where $w\in (e_{i-3}\setminus
 e_{i-4})\setminus V(\mathcal{Q}\cup \mathcal{Q}')$ (the existence of $w$ is guaranteed by Lemma \ref{there is a Pl}). If $h$ is blue, then set
 \begin{eqnarray*}
&&g_1=\{{y'},v_{4i-7},v_{4i-6},v_{4i-5},u_1\},\\
&&g_2=\{{y'},z_1,z_2,v_{4i-4},v_{4i-3}\}.
\end{eqnarray*}
If the edge $g_1$ is blue, Since there is no red copy of $\mathcal{C}_n^5,$
\begin{eqnarray*}
\mathcal{Q}h{\mathcal{Q}'}g_1\{u_1,v_{4i-3},v_{4i-2},v_{4i-1},x\},
\end{eqnarray*}
 is a blue copy of $\mathcal{C}_m^5$. So we may assume that the edge $g_1$ is red. Then the edge $g_2$ is blue (otherwise, $g_1g_2e_i e_{i+1}\ldots e_{n-1}e_1 \ldots e_{i-2}$ is a red copy of $\mathcal{C}_n^5$, a contradiction)  and $\mathcal{Q}h{\mathcal{Q}'}g_2\{v_{4i-3},v_{4i-2},v_{4i-1},u_1,x\}$ is a blue copy of $\mathcal{C}_m^5$.
% \begin{eqnarray*}
%&&\mathcal{C}_1=\mathcal{Q}h{\mathcal{Q}'}\{{y'},v_{4i-7},v_{4i-6},v_{4i-5},u_1\}\{u_1,v_{4i-3},v_{4i-2},v_{4i-1},x\},\\
%&&\mathcal{C}_2=\mathcal{Q}h{\mathcal{Q}'}\{{y'},z_1,z_2,v_{4i-4},v_{4i-3}\}\{v_{4i-3},v_{4i-2},v_{4i-1},u_1,x\}.
%\end{eqnarray*}
%Clearly at least one of $\mathcal{C}_1$ or $\mathcal{C}_2$ is a blue copy of $\mathcal{C}_m^5$.
So we may assume that the edge $h$ is red. Since there is no red copy of $\mathcal{C}_n^5$ the edge $h'=\{x,v_{4i-9},v_{4i-8},v_{4i-7},{y'}\}$ is blue (if not, $hh'e_{i-1}e_i\ldots e_{n-1}e_{1}\ldots e_{i-3}$ is a red copy of $\mathcal{C}_n^5$, a contradiction). Now set
\begin{eqnarray*}
&&h_1=\{x',v_{4i-10},v_{4i-6},v_{4i-5},u_1\},\\
&&h_2=\{x',z_1,z_2,v_{4i-4},v_{4i-3}\}.
\end{eqnarray*}
If the edge $h_1$ is blue, then $\mathcal{Q}h'{\mathcal{Q}'}h_1\{u_1,v_{4i-3},v_{4i-2},v_{4i-1},y\}$ is a blue copy of $\mathcal{C}_m^5$.
 Therefore, we may assume that the edge $h_1$ is red. Since there is no red copy of $\mathcal{C}_n^5,$ the edge $h_2$ is blue and $\mathcal{Q}h'{\mathcal{Q}'}h_2\{v_{4i-3},v_{4i-2},v_{4i-1},u_1,y\}$ is a blue copy of $\mathcal{C}_m^5.$
% So at least one of $\mathcal{C}_3$ or $\mathcal{C}_4$ is a blue copy of $\mathcal{C}_m^5$, where
%\begin{eqnarray*}
%&&\mathcal{C}_3=\mathcal{Q}h'{\mathcal{Q}'}\{x',v_{4i-10},v_{4i-6},v_{4i-5},u_1\}\{u_1,v_{4i-3},v_{4i-2},v_{4i-1},y\},\\
%&&\mathcal{C}_4=\mathcal{Q}h'{\mathcal{Q}'}\{x',z_1,z_2,v_{4i-4},v_{4i-3}\}\{v_{4i-3},v_{4i-2},v_{4i-1},u_1,y\}.
%\end{eqnarray*}
% where $w\in (e_{i-3}\setminus
% e_{i-4})\setminus(\mathcal{Q}\cup\mathcal{Q}')$. At least on of $\mathcal{C}_i$'s, $1\leq i\leq 4$,
%is desired blue $\mathcal{C}_m^5$.\\

Now let $m$ be even. Therefore,  $\ell'=m-4$ and $r\geq 4.$
 Let $u$ be a vertex of $e_{i-5}\setminus e_{i-6}$ so
that $u\notin V(\mathcal{Q}\cup \mathcal{Q}')$ (the existence of $u$ is guaranteed by Lemma \ref{there is a Pl}). Set
 \begin{eqnarray*}
 \overline{W}=\{x,y,x',y',z_1,z_2,u_1\}.
 \end{eqnarray*}
 Using Lemma  \ref{spacial configuration2} there is a good $\varpi_{S}$-configuration in $\mathcal{H}_{\rm blue}$, say $C_1$, with
 end vertices in $\overline{W}$  so that
  $S\subseteq ((e_{i-4}\setminus
\{f_{\mathcal{P},e_{i-4}}\})\cup\{u\})e_{i-3}e_{i-2}$.
  Let $w\in
e_{i-2}\setminus e_{i-3}$ so that $w\notin C_1$.
Using Lemma \ref{spacial configuration2} and since $\ell'$ is maximum,    we may assume that $y'$ and $z_1$ are end vertices of $C_1.$
Let $h=\{z_1,v_{4i-5},v_{4i-6},w,x\}$. If the edge $h$ is blue then
\begin{eqnarray*}
\mathcal{Q}'C_1 h{\mathcal{Q}}\{y,v_{4i-3},v_{4i-2},v_{4i-1},x'\}
\end{eqnarray*}
 is a blue copy of
$\mathcal{C}^5_m$. Otherwise, the edge $h'=\{z_1,z_2,v_{4i-4},v_{4i-3},y\}$ is blue and clearly
\begin{eqnarray*}\mathcal{Q}'C_1 h'{\mathcal{Q}}\{x,v_{4i-5},v_{4i-2},v_{4i-1},x'\}
\end{eqnarray*} is a  copy of $\mathcal{C}^5_m$ in $\mathcal{H}_{\rm blue}$. \\

%one of the following  cases holds:\\
%
%  {\it (i)  $y$ and $x'$ are the end vertices of $C_1$.}
%
%\medskip
%Clearly at least one of $\mathcal{C}'$ or $\mathcal{C}''$ is a blue copy of  $\mathcal{C}^5_m$, where
%\begin{eqnarray*}
%\mathcal{C}'=\mathcal{Q}C_1{\mathcal{Q}'}\{{y'},w,v_{4i-6},v_{4i-5},u_1\}\{u_1,v_{4i-3},v_{4i-2},v_{4i-1},x\},\\
%\mathcal{C}''=\mathcal{Q}C_1{\mathcal{Q}'}\{{y'},z_1,z_2,v_{4i-4},v_{4i-3}\}\{v_{4i-3},v_{4i-2},v_{4i-1},u_1,x\}.
%\end{eqnarray*}
%
%
%   {\it (ii)  $y'$ and $u_1$ are the end vertices of $C_1$.}
%
%\medskip
%Let $h=\{u_1,v_{4i-5},v_{4i-6},w,x\}$. If the edge $h$ is blue then
%\begin{eqnarray*}
%\mathcal{Q}'C_1 h{\mathcal{Q}}\{y,v_{4i-3},v_{4i-2},v_{4i-1},x'\}
%\end{eqnarray*}
% is a blue copy of
%$\mathcal{C}^5_m$. Otherwise the edge $h'=\{u_1,z_2,v_{4i-4},v_{4i-3},y\}$ is blue and clearly
%\begin{eqnarray*}\mathcal{Q}'C_1 h'{\mathcal{Q}}\{x,v_{4i-5},v_{4i-2},v_{4i-1},x'\}
%\end{eqnarray*} is a  copy of $\mathcal{C}^5_m$ in $\mathcal{H}_{\rm blue}$. \\
%
%
%{\it (iii)  $y'$ and $z_1$ are the end vertices of $C_1$.}
%
% The proof of this  case is similar to the proof of  Case $(ii)$. So we emit it here.\\

\item[III.] $|T|=0.$

\medskip
We can clearly observe that $\ell'=2\lfloor\frac{m-1}{2}\rfloor$. If
$m$ is even, then $\ell'=m-2$ and $r\geq 1$. By Lemma \ref{there is a Pl}, there is a vertex
 $w\in
e_{i-2}\setminus e_{i-3}$ so that $w\notin V(\mathcal{Q}\cup\mathcal{Q}')$. Since there is no red copy of $\mathcal{C}_n^5$ and the edge $e$ is red, the edges
\begin{eqnarray*}
&&g=\{y',
v_{4i-3},v_{4i-2},v_{4i-1}, x\},\\
&&g'=\{x',
v_{4i-3},v_{4i-2},v_{4i-1}, y\},
\end{eqnarray*}
are blue.
If the edge $f=\{y,w,v_{4i-6},v_{4i-5},x'\}$ is blue, then $\mathcal{Q}f\mathcal{Q}'g$ is a blue copy of  $\mathcal{C}^5_m$. Therefore, we may assume that the edge $f$ is red.
Thereby  the edge $f'=\{x,v_{4i-5},v_{4i-4},v_{4i},y'\}$ is blue (if not, $ff' e_i\ldots e_{n-1}e_{1}\ldots e_{i-2}$ is a red copy of $\mathcal{C}_n^5$, a contradiction) and $\mathcal{Q}f'\mathcal{Q}'g' $
%\begin{eqnarray*}
%\mathcal{Q}\{y,w,v_{4i-6},v_{4i-5},x'\}\mathcal{Q}'\{y',
%v_{4i-3},v_{4i-2},v_{4i-1}, x\}
%\end{eqnarray*}
%\vspace{-0.3 cm}
% or
%\vspace{-0.2 cm}
% \begin{eqnarray*}
%\mathcal{Q}\{x,v_{4i-5},v_{4i-4},v_{4i},y'\}\mathcal{Q}'\{x',
%v_{4i-3},v_{4i-2},v_{4i-1},y\}
%\end{eqnarray*}
%\vspace{-0.9 cm}
 is desired $\mathcal{C}^5_m$.

If $m$ is odd, then $\ell'=m-1$. Remove the last two edges of $\mathcal{Q}\cup \mathcal{Q}'$ to get two disjoint blue  paths $\overline{\mathcal{Q}}$ and $\overline{\mathcal{Q}'}$ so that $\|\overline{\mathcal{Q}}\cup \overline{\mathcal{Q}'}\|=m-3$ and  $(\overline{\mathcal{Q}}\cup \overline{\mathcal{Q}'})\cap (e_{i-2}\cup e_{i-1})=\emptyset$. We can without loss of generality  assume that $\mathcal{Q}=\overline{\mathcal{Q}}$. First let $\|\overline{\mathcal{Q}'}\|>0$ and $x',y''$ with  $y''\neq y'$ be the end vertices of $\overline{\mathcal{Q}'}$ in $W'$. Since there is no red copy of $\mathcal{C}_n^5$ and the edge $e$ is red, the edges
\begin{eqnarray*}
&&g_1=
\{y',v_{4i-3},v_{4i-2},v_{4i-1},x\},\\
&&g_2=\{y',v_{4i-3},v_{4i-2},v_{4i-1},y\},
\end{eqnarray*}
are blue. Consider the edge $h_1=\{y,v_{4i-11},v_{4i-10},v_{4i-9},x'\}.$ If the edge $h_1$ is blue, then at least one of $\mathcal{C}_1$ or $\mathcal{C}_2$ is the desired blue cycle, where
\begin{eqnarray*}
&&\mathcal{C}_1=\mathcal{Q}h_1\overline{\mathcal{Q}'}\{y'',v_{4i-7},v_{4i-6},v_{4i-5},y'\}g_1,\\
&&\mathcal{C}_2=\mathcal{Q}h_1\overline{\mathcal{Q}'}\{y'',z_1,z_2,v_{4i-4},v_{4i-3}\}g_1.\\
\end{eqnarray*}
So we may assume that the edge $h_1$ is red. Then, since there is no red copy of $\mathcal{C}_n^5,$ the edge $h_2=\{x,v_{4i-9},v_{4i-8},v_{4i-7},y''\}$ is blue. Clearly, at least one of $\mathcal{C}_3$ or $\mathcal{C}_4$ is the desired blue cycle, where
 \begin{eqnarray*}
&&\mathcal{C}_3=\mathcal{Q}h_2\overline{\mathcal{Q}'}\{x',v_{4i-10},v_{4i-6},v_{4i-5},y'\}g_2,\\
&&\mathcal{C}_4=\mathcal{Q}h_2\overline{\mathcal{Q}'}\{x',z_1,z_2,v_{4i-4},v_{4i-3}\}g_2.
\end{eqnarray*}
If $\|\overline{\mathcal{Q}'}\|=0,$ by some discussions similar to the above, we can find a blue copy of $\mathcal{C}_m^5.$
 So
we omit it's proof here.\\
%
%It is easy to check that $\mathcal{Q} h_1 \overline{\mathcal{Q}'} h_2 h_3$ is a blue copy of $\mathcal{C}_m^5$, where

%Let $\overline{Q}$ and
%$\overline{\mathcal{Q}'}$ are two disjoint blue paths obtained by
%removing the last two edges  of $\mathcal{Q}\cup \mathcal{Q}'$. So
%we have one of
% two following cases.\\
%
%{\it i) $\mathcal{Q}=\overline{\mathcal{Q}}$ and
%$\|\overline{\mathcal{Q}'}\|>0$.}
%
%\medskip
%Let $x',\overline{y'}$ be end vertices of
%$\overline{\mathcal{Q}'}$ in $W'$. Set
% \begin{eqnarray*}
%&&\mathcal{C}_1=\mathcal{Q}\{y,v_{4i-11},v_{4i-10},v_{4i-9},x'\}\overline{\mathcal{Q}'}\{\overline{y'},v_{4i-7},v_{4i-6},v_{4i-5},y'\}\{y',v_{4i-3},v_{4i-2},v_{4i-1},x\},\\
%&&\mathcal{C}_2=\mathcal{Q}\{y,v_{4i-11},v_{4i-10},v_{4i-9},x'\}\overline{\mathcal{Q}'}\{\overline{y'},z_1,z_2,v_{4i-4},v_{4i-3}\}\{v_{4i-3},v_{4i-2},v_{4i-1},y',x\},\\
%&&\mathcal{C}_3=\mathcal{Q}\{x,v_{4i-9},v_{4i-8},v_{4i-7},\overline{y'}\}\overline{\mathcal{Q}'}\{x',v_{4i-10},v_{4i-6},v_{4i-5},y'\}\{y',v_{4i-3},v_{4i-2},v_{4i-1},y\},\\
%&&\mathcal{C}_4=\mathcal{Q}\{x,v_{4i-9},v_{4i-8},v_{4i-7},\overline{y'}\}\overline{\mathcal{Q}'}\{x',y',v_{4i-1},v_{4i-2},v_{4i-3}\}\{v_{4i-3},v_{4i-4},v_{4i-5},z_1,y\}.
%\end{eqnarray*}
%Clearly, at least on of $\mathcal{C}_i$'s, $1\leq i\leq 4$,
%is desired blue $\mathcal{C}_m^5$.\\

\end{itemize}

\medskip
\noindent{\it Subcase $2$}. $\|\mathcal{Q}'\|=0$.

\medskip
% First, assume that $\|\mathcal{Q}'\|=0$.
\noindent Let $x$ and $y$ be the
end vertices of $\mathcal{Q}$ in $W'$ and $T=W_0\setminus W'$. Using Lemma \ref{there is a Pl}, we have one of the following cases:

%let $\mathcal{Q}'$ be a blue path of length $m-1$ with end
%vertices $x,y'$ in $W'$ by removing the last two edges of
%$\mathcal{Q}$. So
%$\mathcal{Q}'\{x,v_{4i-3},v_{4i-2},v_{4i-1},y'\}$ is the desired
%blue cycle.

\begin{itemize}
\item[I.] $|T|\geq 4$.

\medskip
Since $\ell'\leq 2\lfloor\frac{m-1}{2}\rfloor-6$ and $\ell'=\frac{2}{3}(n-2-r),$
we have $r\geq 3$. So this subcase
does not occur by Lemma \ref{there is a Pl}.\\

\item[II.] $|T|=3.$

\medskip
Let $T=\{u_1,u_2,u_3\}$. Clearly
$\ell'=2\lfloor\frac{m-1}{2}\rfloor-4$. First let $m$ be odd. Therefore,  $\ell'=m-5.$ Since $\ell'=\frac{2}{3}(n-2-r)$ and $n\geq \lfloor\frac{3m}{2}\rfloor,$  we have $r\geq 5.$ By Lemma \ref{there is a Pl}, there is a vertex  $v\in e_{i-6}\setminus e_{i-7}$ so that $v\notin
 V(\mathcal{Q})$. Since $\mathcal{P}$ is maximal w.r.t.
 $\overline{W}=\{x,y,z_1,z_2,u_1,u_2,u_3\}$, using Lemma  \ref{spacial configuration2},
 %and
% $v\in e_{i-6}\setminus e_{i-7}$ be a vertex so that $v\notin
% V(\mathcal{Q})$ (the existence of $v$ is guaranteed by Lemma \ref{there is a Pl}).
%Using Lemma \ref{spacial configuration2},
 there is a good $\varpi_{S}$-configuration, say $C_1=fg$, in $\mathcal{H}_{\rm blue}$ with end vertices $\overline{x}\in f$ and $\overline{y}\in g$ in $\overline{W}$ and $S\subseteq (e_{i-5}\setminus\{f_{\mathcal{P},e_{i-5}}\}\cup\{v\})e_{i-4}e_{i-3}$.
  By Lemma  \ref{spacial configuration2}, there is a vertex
  of $e_{i-3}\setminus e_{i-4},$ say $w,$ so that  $w\notin V(C_1)$. Moreover there are two subsets $W_1$ and $W_2$ of $\overline{W}$ with $|W_1|\geq 4$ and $|W_2|\geq 3$ so that for every distinct vertices $\overline{x}'\in W_1$ and $\overline{y}'\in W_2,$ the path
  $C_1'=((f\setminus \{\overline{x}\})\cup\{\overline{x}'\})((g\setminus \{\overline{y}\})\cup\{\overline{y}'\})$ is also a good $\varpi_{S}$-configuration in $\mathcal{H}_{\rm blue}$ with end vertices $\overline{x}'$ and $\overline{y}'$ in $\overline{W}$. Since $|W_1|\geq 4$ and $\ell'$ is maximum,
%
%
%  By Lemma \ref{our conjecture2}, each vertex of $\overline{W}$ with the exception of at most three can be considered as an end vertex of $C_1$.
% Since $\ell'$ is maximum,
%
 by symmetry,  we may
assume that   $y$ and $z_1$ are  the end vertices of $C_1$  in $\overline{W}$.
  If the edge $f=\{u_1,w,v_{4i-10},v_{4i-9},x\}$ is blue, then set
\begin{eqnarray*}
&&g_1=\{u_2,v_{4i-5},v_{4i-6},v_{4i-7},u_1\},\\
&&g_2=\{z_1,v_{4i-5},v_{4i-4},u_3,v_{4i-3}\}.
\end{eqnarray*}
First let  the edge $g_1$ is blue. Since there is no red copy of $\mathcal{C}_n^5$ and the edge $e$ is red, the cycle  $g_1f\mathcal{Q}C_1\{z_1,v_{4i-3},v_{4i-2},u_3,u_2\}$ is a blue copy of $\mathcal{C}_m^5.$
So we may  suppose that the edge
 $g_1$ is red.  Therefore,  the edge $g_2$ is blue ( otherwise, $g_1 g_2 e_i e_{i+1} \ldots e_{n-1} e_1 \ldots e_{i-2}$ is a red copy of $\mathcal{C}_n^5,$ a contradiction to our assumption). Again, since there is no red copy of $\mathcal{C}_n^5$ and the edge $e$ is red, the edge $g_3=\{v_{4i-3},v_{4i-2},v_{4i-1},u_2,u_1\}$ is blue (otherwise, $g_3ee_{i+1}\ldots e_{n-1}e_1\ldots e_{i-1}$ is a red copy of $\mathcal{C}_n^5,$ a contradiction) and
  $f\mathcal{Q}C_1g_2g_3$ is a blue copy of $\mathcal{C}_m^5$.
%\begin{eqnarray*}
%&&\mathcal{C}_1=\{u_2,v_{4i-5},v_{4i-6},v_{4i-7},u_1\}f\mathcal{Q}C_1\{z_1,v_{4i-3},v_{4i-2},u_3,u_2\},\\
%&&\mathcal{C}_2=f\mathcal{Q}C_1\{z_1,v_{4i-5},v_{4i-4},u_3,v_{4i-3}\}\{v_{4i-3},v_{4i-2},v_{4i-1},u_2,u_1\}.
%\end{eqnarray*}
%Clearly at least one of $\mathcal{C}_1$ or  $\mathcal{C}_2$ is a blue copy of $\mathcal{C}_m^5$.
Now suppose that the edge $f$ is red. Then the edge
  $f'=\{u_3,v_{4i-7},v_{4i-8},v_{4i-9},z_1\}$ is blue.
%  (otherwise, $ff'e_{i-1}e_{i}\ldots e_{n-1}e_1 \ldots e_{i-3}$ is a red copy of $\mathcal{C}_n^5,$ a contradiction).
  Let
 \begin{eqnarray*}
&&h_1=\{u_2,v_{4i-5},v_{4i-6},v_{4i-10},u_3\},\\
&&h_2=
\{x,u_1,v_{4i-5},v_{4i-4},v_{4i-3}\}.
\end{eqnarray*}
 If the edge $h_1$ is blue, then  $h_1f'
C_1\mathcal{Q}\{x,v_{4i-3},v_{4i-2},v_{4i-1},u_2\}$ is a blue copy of $\mathcal{C}_m^5.$ So suppose  that the edge $h_1$ is red. Therefore,  the edge $h_2$ is blue.
%(otherwise, $h_1h_2e_ie_{i+1}\ldots e_{n-1}e_1\ldots e_{i-2}$ is a red copy of $\mathcal{C}_n^5,$ a contradiction).
Again, since there is no red copy of $\mathcal{C}_n^5$ and the edge $e$ is red,   the cycle $f'C_1\mathcal{Q}
h_2\{v_{4i-3},v_{4i-2},v_{4i-1},u_2,u_3\}$ is a blue copy of $\mathcal{C}_m^5.$

%  and at least one of $\mathcal{C}_3$ or $\mathcal{C}_4$ is a blue copy of $\mathcal{C}_m^5$, where
%\begin{eqnarray*}
%&&\mathcal{C}_3=\{u_2,v_{4i-5},v_{4i-6},v_{4i-10},u_3\}g
%C_1\mathcal{Q}\{x,v_{4i-3},v_{4i-2},v_{4i-1},u_2\},\\
%&&\mathcal{C}_4=gC_1\mathcal{Q}
%\{x,u_1,v_{4i-5},v_{4i-4},v_{4i-3}\}\{v_{4i-3},v_{4i-2},v_{4i-1},u_2,u_3\}.
%\end{eqnarray*}
%\vspace{-0.9 cm}
% One can check that at least one of the
% $\mathcal{C}_i$'s, $1\leq i\leq 4,$ is a blue copy of $\mathcal{C}_m^5$.

\medskip
Now, let $m$ be even. Hence, we have  $\ell'=m-6$ and $r\geq 7$.  Let $ \overline{W} = \{x, y, z_1, z_2, u_1, u_2, u_3\}$ and $v\in e_{i-8}\setminus
e_{i-9}$ be a vertex so that $v\notin
 V(\mathcal{Q})$ (the existence of $v$ is guaranteed by Lemma \ref{there is a Pl}).
Using Lemma
\ref{spacial configuration2}, there is a good $\varpi_{S_1}$-configuration, say $C_1$, in $\mathcal{H}_{\rm blue}$ with end vertices in $\overline{W}$ so that $S_1\subseteq (e_{i-7}\setminus\{f_{\mathcal{P},e_{i-7}}\}\cup\{v\})e_{i-6}e_{i-5}$ and at least one of the vertices of $e_{i-5}\setminus e_{i-6},$ say $w,$ is not in $C_1$. By an argument similar to the case that $m$ is odd,
 we may assume that $y$ and $z_1$ are  the end vertices of $C_1$  in $\overline{W}$. Now set $\widetilde{W}=\overline{W}\setminus\{y\}=\{x,z_1,z_2,u_1,u_2,u_3\}.$ Again, using Lemma
\ref{spacial configuration2}, there is a good $\varpi_{S_2}$-configuration, say $C_2$, in $\mathcal{H}_{\rm blue}$ with end vertices in $\widetilde{W}$ so that $S_2\subseteq (e_{i-4}\setminus\{f_{\mathcal{P},e_{i-4}}\}\cup\{w\})e_{i-3}e_{i-2}$ and at least one of the vertices of $e_{i-2}\setminus e_{i-3},$ say $w',$ is not in $C_2$. By the properties of Lemma \ref{spacial configuration2} and since $\ell'$ is maximum, we may suppose  that $x$ and  $z_2$ or $z_2$ and $u_1$ are end vertices of $C_2$ in $\widetilde{W}.$
 %  there are  two blue configurations,
%say $C_1$ and $C_2$, between
%$((e_{i-7}\setminus\{f_{\mathcal{P},e_{i-7}}\})\cup\{w\})e_{i-6}\ldots
%e_{i-2}$ and $\overline{W}=T\cup \{x,y,z_1,z_2\}$ with end vertices in $W$
%where $w$ is a vertex of $e_{i-8}\setminus e_{i-9}$ so that
%$w\notin \mathcal{Q}$. Without loss of generality   we may assume
%that $\mathcal{Q}C_1C_2$ is one path or $\mathcal{Q}C_1$ and $C_2$
%are two blue paths with end vertices  in $\overline{W}$.
It is not difficult to show that in each cases
%, by analogous arguments to the
%above subcases, to show that
 there is a blue copy of
$\mathcal{C}^5_m$. Here, for
abbreviation, we omit the proof.\\

\item[III.] $|T|=2.$

\medskip
Let $T=\{u_1,u_2\}$. Since $\ell'=2\lfloor\frac{m-1}{2}\rfloor-2$,
for odd $m$ we have  $\ell'=m-3$ and $r\geq 2$.
By Lemma \ref{there is a Pl}, there is a vertex  $w\in e_{i-3}\setminus e_{i-4}$ so that
$w\notin V(\mathcal{Q}).$ If the edge $f=\{y,w,v_{4i-10},v_{4i-9},u_1\}$ is blue, then set
\begin{eqnarray*}
&&g_1=\{u_1,z_1,z_2,v_{4i-4},v_{4i-3}\},\\
&&g_2=\{u_1,v_{4i-7},v_{4i-6},v_{4i-5},u_2\}.
\end{eqnarray*}
 Since there is no red copy of $\mathcal{C}_n^5,$ at least one of the edges $g_1$ or $g_2$, say $g',$  is  blue  (otherwise, $g_2g_1 e_{i}e_{i+1}\ldots e_{n-1}e_1\ldots e_{i-2}$ is a red copy of $\mathcal{C}_n^5$). Now, since that edge $e$ is red, the cycle  $\mathcal{C}_1=\mathcal{Q}fg'\{v_{4i-3},v_{4i-2},v_{4i-1},u_2,x\}$  is a blue copy of $\mathcal{C}_m^5$.
%\begin{eqnarray*}
%&&\mathcal{C}_1=\mathcal{Q}fg'\{v_{4i-3},v_{4i-2},v_{4i-1},u_2,x\},\\
%&&\mathcal{C}_2=\mathcal{Q}f\{u_2,v_{4i-3},v_{4i-2},v_{4i-1},x\},
%\end{eqnarray*}
%where $g_1=\{u_1,z_1,z_2,v_{4i-4},v_{4i-3}\}$ and $g_2=\{u_1,v_{4i-7},v_{4i-6},v_{4i-5},u_2\}.$
 If the edge $f$ is red, then
 % since there is no red copy of $\mathcal{C}_n^5,$
  the edge
 $g=\{x,v_{4i-9},v_{4i-8},v_{4i-7},u_2\}$ is blue and at least one of $\mathcal{C}_3$ or $\mathcal{C}_4$ is a blue copy of $\mathcal{C}_m^5$, where
\begin{eqnarray*}
&&\mathcal{C}_3=\mathcal{Q}g\{u_2,u_1,v_{4i-1},v_{4i-2},v_{4i-3}\}\{v_{4i-3},v_{4i-4},v_{4i-5},z_1,y\},\\
&& \mathcal{C}_4=\mathcal{Q}g\{u_2,v_{4i-10},v_{4i-6},v_{4i-5},u_1\}\{u_1,v_{4i-3},v_{4i-2},v_{4i-1},y\}.
\end{eqnarray*}
% Obviously at least one of $\mathcal{C}_i$'s, $1\leq i\leq 4,$ is a  blue copy of $\mathcal{C}_m^5$.

Now, we may assume that $m$ is even. Consequently, $\ell'=m-4$ and
$r\geq 4$. Let $\overline{W}=\{x,y,z_1,z_2,u_1,u_2\}$ and $u$ be a vertex of $e_{i-5}\setminus e_{i-6}$ so
that $u\notin V(\mathcal{Q})$.
 Using Lemma \ref{spacial configuration2}, there is a good $\varpi_{S}$-configuration, say $C_1$, in $\mathcal{H}_{\rm blue}$  with end vertices in $\overline{W}$ so that $S\subseteq ((e_{i-4}\setminus
\{f_{\mathcal{P},e_{i-4}}\})\cup\{u\})e_{i-3}e_{i-2}$.
  Let $w\in
e_{i-2}\setminus e_{i-3}$ so that $w\notin C_1$. Since $\ell'$ is maximum, then $u_1,u_2$ or $u_j,x$ or $u_j,y$ for $1\leq j \leq 2$ can not be end vertices of $C_1.$
By symmetry  we may assume that one of the following  cases holds:\\

 {\it $(i)$  $y$ and $z_1$ are the end vertices of $C_1$ $($when $y,z_2$ or $x,z_1$ or $x,z_2$ are end vertices of $C_1$ the proof is similar to this case$)$.}

 \medskip
If the edge $f=\{u_1,v_{4i-5},v_{4i-6},w,x\}$ is blue, then $\mathcal{Q}C_1\{z_1,v_{4i-2},v_{4i-3},u_2,u_1\}f$ is a blue copy of $\mathcal{C}_m^5$. So we may suppose  that the edge $f$ is red. Since there is no red copy of $\mathcal{C}_n^5,$ the edge $f'=\{z_1,u_2,v_{4i-5},v_{4i-4},v_{4i-3}\}$ is blue and
\begin{eqnarray*}\mathcal{Q}C_1f'\{v_{4i-3},v_{4i-2},v_{4i-1},u_1,x\},
\end{eqnarray*}
  is a blue copy of  $\mathcal{C}_m^5.$\\
%\begin{eqnarray*}
%&&\mathcal{C}'=\mathcal{Q}C_1\{z_1,v_{4i-2},v_{4i-3},u_2,u_1\}\{u_1,v_{4i-5},v_{4i-6},w,x\},\\
%&&\mathcal{C}''=\mathcal{Q}C_1\{z_1,u_2,v_{4i-5},v_{4i-4},v_{4i-3}\}\{v_{4i-3},v_{4i-2},v_{4i-1},u_1,x\}.
%\end{eqnarray*}

% {\it (ii)  $y$ and $u_1$ are the  end vertices of $C_1$ (when $y,u_2$ or $x,u_1$ or $x,u_2$ are end vertices of $C_1$ the proof is similar to this case).}
%
% \medskip
%
%Set $f=\{u_1,w,v_{4i-6},v_{4i-5},u_2\}$. If the edge $f$ is blue, then
% \begin{eqnarray*}
%\mathcal{Q}C_1f\{u_2,v_{4i-3},v_{4i-2},v_{4i-1},x\},
%\end{eqnarray*}
% is a blue copy of $\mathcal{C}_m^5$. Otherwise, since there is no red copy of $\mathcal{C}_n^5,$  the edge $f'=\{v_{4i-3},v_{4i-4},v_{4i-5},z_1,x\}$ is blue and $\mathcal{Q}C_1\{u_1,u_2,v_{4i-1},v_{4i-2},v_{4i-3}\}f'$ is a blue copy of $\mathcal{C}_m^5$.\\

% \begin{eqnarray*}
%&&\mathcal{C}'=\mathcal{Q}C_1\{u_1,w,v_{4i-6},v_{4i-5},u_2\}\{u_2,v_{4i-3},v_{4i-2},v_{4i-1},x\},\\
%&&\mathcal{C}''=\mathcal{Q}C_1\{u_1,u_2,v_{4i-1},v_{4i-2},v_{4i-3}\}\{v_{4i-3},v_{4i-4},v_{4i-5},z_1,x\}.
%\end{eqnarray*}
%Clearly at least one of $\mathcal{C}'$ or $\mathcal{C}''$
%is a blue copy of  $\mathcal{C}_m^5$.\\

 {\it $(ii)$ $u_2$ and $z_1$  are the   end
vertices of $C_1$ $($when $u_2,z_2$ or $u_1,z_1$ or $u_1,z_2$ are end vertices of $C_1$ the proof is similar to this case$)$.}

\medskip
If the edge $f=\{y,w,v_{4i-6},v_{4i-5},u_2\}$ is blue, then $\mathcal{Q}fC_1\{z_1,v_{4i-3},v_{4i-2},u_1,x\}$ is a blue copy of $\mathcal{C}_m^5$. Otherwise, $\mathcal{Q}f'C_1\{u_2,z_2,v_{4i-2},v_{4i-3},y\}$ is our desired cycle, where $f'=\{x,v_{4i-5},v_{4i-4},v_{4i-1},z_1\}.$\\

%In this case, at least one of  $\mathcal{C}'$ or $\mathcal{C}''$
%is a blue copy of  $\mathcal{C}_m^5$, where
%\begin{eqnarray*}
%&&\mathcal{C}'=\mathcal{Q}\{y,w,v_{4i-6},v_{4i-5},u_2\}C_1\{z_1,v_{4i-3},v_{4i-2},u_1,x\},\\
%&&\mathcal{C}''=\mathcal{Q}\{x,v_{4i-5},v_{4i-4},v_{4i-1},z_1\}C_1\{u_2,z_2,v_{4i-2},v_{4i-3},y\}.
%\end{eqnarray*}

\item[IV.] $|T|=1.$

\medskip
Let $T=\{u_1\}$. One can easily check that
$\ell'=2\lfloor\frac{m-1}{2}\rfloor$. If $m$ is odd, then  $\ell'=m-1$.
Let $w$ be a vertex of $e_{i-1}\setminus e_{i-2}$ so that $w\notin
V(\mathcal{Q})$(the existence of $w$ is guaranteed by Lemma \ref{there is a Pl}). Since there is no red copy of $\mathcal{C}^5_n$ and the edge $e$ is red, the cycle $\mathcal{Q}\{y,u_1,v_{4i-1},w,x\}$
is a blue copy of $\mathcal{C}^5_m$.

Now, we may assume that $m$ is even. So
$\ell'=m-2$ and $r\geq 1$. Let $w$ be a vertex of $e_{i-2}\setminus
e_{i-3}$ so that $w\notin V(\mathcal{Q})$. Since there is no red copy of $\mathcal{C}^5_n$ at least one of the edges
\begin{eqnarray*}
&&g_1=\{y,w,v_{4i-6},v_{4i-5},u_1\},\\
&&g_2=\{x,v_{4i-5},v_{4i-4},z_1,v_{4i-3}\},
\end{eqnarray*}
is blue (if not, $g_1g_2 e_ie_{i+1}\ldots e_{n-1}e_1\ldots e_{i-2}$ form a red copy of $\mathcal{C}^5_n$). If the edge $g_1$ is blue, then
\begin{eqnarray*}
\mathcal{Q}g_1\{u_1,v_{4i-3},v_{4i-2},v_{4i-1},x\}
\end{eqnarray*}
is a blue copy of $\mathcal{C}^5_m.$ Otherwise,
\begin{eqnarray*}
\mathcal{Q}g_2\{v_{4i-3},v_{4i-2},v_{4i-1},
u_1,y\}
\end{eqnarray*}
 is a copy of $\mathcal{C}^5_m$ in
 $\mathcal{H}_{\rm blue}$.
%Now by simple arguments, similar to Subcase $1$, we have a blue
%copy of $\mathcal{C}^5_m$.
%

\item[V.] $|T|=0.$

\medskip
One can easily check that $\ell'=2\lfloor\frac{m-1}{2}\rfloor+2$.
Remove the last
two edges of $\mathcal{Q}$  to get two disjoint blue paths $\overline{\mathcal{Q}}$ and $\overline{\mathcal{Q}'}$ so that $\|\overline{\mathcal{Q}}\cup\overline{\mathcal{Q}'}\|=2\lfloor\frac{m-1}{2}\rfloor$ and
$(\overline{\mathcal{Q}}\cup \overline{\mathcal{Q}'})\cap((e_{i-3}\setminus\{f_{\mathcal{P},e_{i-3}}\})\cup e_{i-2}\cup e_{i-1})=\emptyset$. By an argument similar to the case $|T|=0$ of Subcase $1$ and the case $|T|=1$ of this subcase, we can find a blue copy of $\mathcal{C}^5_m.$

\end{itemize}

\bigskip
\noindent \textbf{Case 2. }For every edge
$e_i=\{v_{4i-3},v_{4i-2},v_{4i-1},v_{4i},v_{4i+1}\}$, $1\leq i\leq
n-1$, and every vertices $z_1,z_2\in W$ and $v'\in e_{i+1}\setminus\{l_{\mathcal{C},e_{i+1}}\}$ (also $v'\in e_{i-1}\setminus\{f_{\mathcal{C},e_{i-1}}\}$)
 the edge $\{v_{4i-1},v_{4i},v',z_1,z_2\}$
(also  the edge $\{v',v_{4i-2},v_{4i-1},z_1,z_2\}$) is blue.

\medskip
 Let $W=\{u_1,u_2,\ldots,u_{\lfloor\frac{m-1}{2}\rfloor+4}\}$.
First assume that $m=7$.
%For fixed $3\leq m\leq 7$,
 Set
\begin{eqnarray*}
h_i=\left\lbrace \begin{array}{ll}
(e_i\setminus \{v_{4i},v_{4i+1}\})\cup \{u_i,u_{i+1}\}
&\mbox{for $1\leq i \leq 6$},\vspace{.5 cm}\\
(e_i\setminus \{v_{4i},v_{4i+1}\})\cup \{u_i,u_{1}\} &\mbox{for $i=7$}.
\end{array}\right.
 \end{eqnarray*}
Since $h_i$'s, $1\leq i \leq 7$, are blue, then   $h_1h_2\ldots h_7$ is a blue copy of  $\mathcal{C}^5_7$. Therefore, we may suppose  that $m\geq 8$.
Let $\mathcal{P}=e_1e_2\ldots e_{n-3}$ and $W_0=W\setminus \{u_1\}.$
Clearly $|W_0|\geq 6$. Since there is no red copy
of $\mathcal{C}^5_n$, $\mathcal{P}$ is a maximal path w.r.t.
$W_0$. Use Lemma \ref{there is a Pl} to obtain   two
disjoint blue paths $\mathcal{Q}$ and $\mathcal{Q}'$
%with
%$V(\mathcal{Q})\cap V(\mathcal{Q}')=\emptyset$
% and $\|\mathcal{Q}\cup \mathcal{Q}'\|=\ell'$
 between
 $\overline{\mathcal{P}}$, the path obtained from $\mathcal{P}$ by deleting the last $r$ edges for
 some $r\geq 0$  and $W'\subseteq W_0$ with the mentioned properties.
 %Let   $\mathcal{Q}$ and $\mathcal{Q}'$ be such paths  so that $\ell'=\|\mathcal{Q}\cup \mathcal{Q}'\|$ is maximum.
   Consider the paths $\mathcal{Q}$ and $\mathcal{Q}'$ with $\|\mathcal{Q}\|\geq \|\mathcal{Q}'\|$ so that $\ell'=\|\mathcal{Q}\cup \mathcal{Q}'\|$ is maximum. Among these paths, choose $\mathcal{Q}$ and $\mathcal{Q}'$, where $\|\mathcal{Q}\|$ is maximum.   Since $\|\mathcal{P}\|=n-3,$ by Lemma \ref{there is a Pl}, we have $\ell'=\frac{2}{3}(n-3-r)$.\\

\noindent {\it Subcase $1$}. $\|\mathcal{Q}'\|\neq 0$.

\medskip
 %Now we may assume that  $\|\mathcal{Q}'\|\neq 0$.
\noindent  Set $T=W_0\setminus W'.$ Let  $x,y$ and $x',y'$ be the end
vertices of $\mathcal{Q}$ and $\mathcal{Q}'$ in $W'$,
respectively. Using Lemma \ref{there is a Pl} we have one of
the following cases:

\begin{itemize}

\medskip
\item[I.] $|T|\geq 3.$

\medskip
In this case we have $\ell'\leq 2\lfloor\frac{m-1}{2}\rfloor-4$ and so
 $r\geq 4$. Therefore,  this case
does not occur by Lemma \ref{there is a Pl}.

 \medskip
\item[II.] $|T|=2.$

\medskip
Let $T=\{u_2,u_3\}$.
In this case, $\ell'=2\lfloor\frac{m-1}{2}\rfloor-2$. If $m$ is even, then $\ell'=m-4$ and $r\geq 3$. It is impossible  by Lemma \ref{there is a Pl}. So we may assume that $m$ is odd. Therefore, $\ell'=m-3$ and $r\geq 1$. Based on our  assumptions,  the edges $h_1,h_2$ and $h_3$ are blue, where
\begin{eqnarray*}
&&h_1=(e_{n-3}\setminus \{v_{4(n-3)-3},v_{4(n-3)-2}\})\cup\{y,x'\},\\
&&h_2=(e_{n-2}\setminus \{v_{4(n-2)-3},v_{4(n-2)-2}\})\cup\{y',u_2\},\\
&&h_3=(e_{n-1}\setminus \{v_{4(n-1)},v_1\})\cup\{u_3,x\}.
\end{eqnarray*}
Thereby $\mathcal{Q}h_1\mathcal{Q}' h_2h_3$ is our desired blue $\mathcal{C}_m^5$.

\medskip
\item[III.] $|T|=1.$

\medskip
By symmetry  we may assume that  $T=\{u_2\}$. Clearly $\ell'=2\lfloor\frac{m-1}{2}\rfloor$. If
 $m$ is odd, then  $\ell'=m-1$. Remove the last two edges of $\mathcal{Q}\cup \mathcal{Q}'$ to get two disjoint blue  paths $\overline{\mathcal{Q}}$ and $\overline{\mathcal{Q}'}$ so that $V(\overline{\mathcal{Q}}\cup \overline{\mathcal{Q}'})\cap (e_{n-4}\cup e_{n-3})=\emptyset$. We can without loss of generality  assume that $\mathcal{Q}=\overline{\mathcal{Q}}$. First let $\|\overline{\mathcal{Q}'}\|>0$ and $x',y''$ with  $y''\neq y'$ be the end vertices of $\overline{\mathcal{Q}'}$ in $W'$. It is easy to check that $\mathcal{Q} h_1 \overline{\mathcal{Q}'} h_2 h_3$ is a blue copy of $\mathcal{C}_m^5$, where
 \begin{eqnarray*}
&&h_1=(e_{n-3}\setminus \{v_{4(n-3)-3},v_{4(n-3)-2}\})\cup\{y,x'\},\\
&&h_2=(e_{n-2}\setminus \{v_{4(n-2)-3},v_{4(n-2)-2}\})\cup\{y'',y'\},\\
&&h_3=(e_{n-1}\setminus \{v_{4(n-1)},v_1\})\cup\{u_2,x\}.
\end{eqnarray*}
So we may assume that $\|\overline{\mathcal{Q}'}\|=0$. Clearly $\mathcal{Q} h'_1  h'_2 h'_3$ is a blue copy of $\mathcal{C}_m^5$, where
\begin{eqnarray*}
&&h'_1=(e_{n-3}\setminus \{v_{4(n-3)-3},v_{4(n-3)-2}\})\cup\{y,x'\},\\
&&h'_2=(e_{n-2}\setminus \{v_{4(n-2)-3},v_{4(n-2)-2}\})\cup\{x',y'\},\\
&&h'_3=(e_{n-1}\setminus \{v_{4(n-1)},v_1\})\cup\{u_2,x\}.
\end{eqnarray*}

Now let $m$ be even. Therefore,  $\ell'=m-2$ and $r\geq 0.$
Let  $w$ be a vertex of $e_{n-3}\setminus e_{n-4}$ so
that $w\notin V(\mathcal{Q}\cup \mathcal{Q}')$ (the existence of $w$ is guaranteed by Lemma \ref{there is a Pl}). By the assumption,  the edges $h$ and $h'$ are blue, where
\begin{eqnarray*}
&&h=(e_{n-1}\setminus\{v_{4(n-1)},v_{1}\})\cup\{y',x\},\\
&&h'=(e_{n-2}\setminus\{v_{4(n-2)-3},v_{4(n-2)},v_{4(n-2)+1}\})\cup\{w,y,x'\}.
 \end{eqnarray*}
Thereby $\mathcal{Q}h'\mathcal{Q}' h$ is a blue copy of $\mathcal{C}_m^5$.
%\begin{eqnarray*}
%\mathcal{Q}h\mathcal{Q}' h'
%\end{eqnarray*}

\end{itemize}

\medskip
\noindent{\it Subcase $2$.}
 $\|\mathcal{Q}'\|=0$.

 \medskip
% First, assume that $\|\mathcal{Q}'\|=0$.
\noindent Let $x$ and $y$ be the
end vertices of $\mathcal{Q}$ in $W'$ and  $T=W_0\setminus W'$.  Using Lemma \ref{there is a Pl} we have the following
 cases:

\begin{itemize}

\medskip
\item[I.] $|T|\geq 4$.

\medskip
In this case, we have  $\ell'=2\lfloor\frac{m-1}{2}\rfloor-4$ and  $r\geq 4$.
 So this subcase
does not hold by Lemma \ref{there is a Pl}.

\medskip
\item[II.] $|T|=3.$

\medskip
Let  $T=\{u_2,u_3,u_4\}$. Clearly
$\ell'=2\lfloor\frac{m-1}{2}\rfloor-2$. First let $m$ be odd. Then
$\ell'=m-3$ and $r\geq 1$.  Set
\begin{eqnarray*}
&&h_1=(e_{n-2}\setminus \{v_{4(n-2)},v_{4(n-2)+1}\})\cup\{y,u_1\},\\
&&h_2=(e_{n-2}\setminus \{v_{4(n-2)-3},v_{4(n-2)-2}\})\cup\{u_2,u_3\},\\
&&h_3=(e_{n-1}\setminus \{v_{4(n-1)},v_1\})\cup\{u_4,x\}.
\end{eqnarray*}
%\vspace{-0.9 cm}
 Clearly  $\mathcal{Q}h_1h_2h_3$ is a blue copy of $\mathcal{C}_m^5$.

\medskip
Now, let $m$ be even. Hence $\ell'=m-4$ and $r\geq 3$.  It is easy  to see that $\mathcal{Q}h'_1h'_2h'_3h'_4$ is a blue copy of
$\mathcal{C}^5_m$, where
\begin{eqnarray*}
&&h'_1=(e_{n-4}\setminus \{v_{4(n-4)},v_{4(n-4)+1}\})\cup\{y,u_1\},\\
&&h'_2=(e_{n-3}\setminus \{v_{4(n-3)},v_{4(n-3)+1}\})\cup\{u_1,u_2\},\\
&&h'_3=(e_{n-2}\setminus \{v_{4(n-2)},v_{4(n-2)+1}\})\cup\{u_2,u_3\},\\
&&h'_4=(e_{n-1}\setminus \{v_{4(n-1)},v_1\})\cup\{u_3,x\}.
\end{eqnarray*}

\medskip
\item[III.] $|T|=2.$

\medskip
Let $T=\{u_2,u_3\}$. Clearly  $\ell'=2\lfloor\frac{m-1}{2}\rfloor$. If  $m$ is odd, then
 $\ell'=m-1$. By the assumption, the edge $h=(e_{n-1}\setminus \{v_{4(n-1)},v_1\})\cup\{x,y\}$ is blue and so $\mathcal{Q}h$ is a blue copy of $\mathcal{C}_m^5$.

Now, we may assume that $m$ is even. So $\ell'=m-2$. It is easy to see that $\mathcal{Q}h'h$
  is a copy of  $\mathcal{C}_m^5$ in $\mathcal{H}_{\rm blue}$, where
\begin{eqnarray*}
&&h=(e_{n-1}\setminus \{v_{4(n-1)},v_1\})\cup\{u_2,x\},\\
&&h'=(e_{n-2}\setminus \{v_{4(n-2)-3},v_{4(n-2)-2}\})\cup\{y,u_3\}.
\end{eqnarray*}

\end{itemize}

 }\end{proof}

%%%%%%%%%%%%%%%%%%%%%%%%%%%%%%%%%%%%%%%%%%%%%%%%%%%%%%%%%%%%%%%%%%%%%%%%%%%%%%%%%%%%%%%%%%%%%%%%%%%%%%
We shall use Theorem  \ref{R(C3,C4)} and Lemmas \ref{pm-1 implies pn-1}, \ref{Cn-1 implies small Cm} and
\ref{Cn-1 implies cm} to prove the following main theorem.
%%%%%%%%%%%%%%%%%%%%%%%%%%%%%%%%%%%%%%%%%%%%%%%%%%%%%%%%%%%%%%%%%%%%%%%%%%%%%%%%%%%%%%%%%%%%%%%%%%%%%%

\bigskip
\begin{theorem}\label{main theorem,k=5}
For every $n\geq \Big\lfloor\frac{3m}{2}\Big\rfloor$,
$$R(\mathcal{C}^5_n,\mathcal{C}^5_m)=4n+\Big\lfloor\frac{m-1}{2}\Big\rfloor.$$
\end{theorem}
%%%%%%%%%%%%%%%%%%%%%%%%%%%%
%%%%%%%%%%%%%%%%%%%%%%%%%%%%
\begin{proof}{We give a proof by induction on $m+n$.  Using Theorem
\ref{R(C3,C4)} the statement of this theorem holds for $m=3.$  Let $m\geq 4,$ $n\geq \Big\lfloor\frac{3m}{2}\Big\rfloor$ and
$\mathcal{H}=\mathcal{K}^5_{4n+\lfloor\frac{m-1}{2}\rfloor}$ be 2-edge colored
red and blue with no red copy of
$\mathcal{C}^5_n$ and no blue copy of $\mathcal{C}^5_m$.  Consider
the following cases:

\bigskip
\noindent \textbf{Case 1. } $n=
\Big\lfloor\frac{3m}{2}\Big\rfloor$.\\
By induction hypothesis,
$$R(\mathcal{C}^5_{n-1},\mathcal{C}^5_{m-1})=
4(n-1)+\Big\lfloor\frac{m-2}{2}\Big\rfloor<
4n+\Big\lfloor\frac{m-1}{2}\Big\rfloor.$$ Therefore,  there is a copy of
$\mathcal{C}^5_{n-1}\subseteq \mathcal{H}_{\rm red}$ or a copy of
$\mathcal{C}^5_{m-1}\subseteq \mathcal{H}_{\rm blue}$. If we have
a red copy of $\mathcal{C}^5_{n-1}$, then by Lemma \ref{Cn-1 implies small Cm} or
 \ref{Cn-1
implies cm}  we have a copy of  $\mathcal{C}^5_{m}\subseteq \mathcal{H}_{\rm
blue}$. So, we may suppose that there is a blue copy of
$\mathcal{C}^5_{m-1}$.
%then Since $n-1= \Big\lfloor\frac{5(m-1)}{4}\Big\rfloor$,
 Lemma \ref{pm-1 implies pn-1} implies that $\mathcal{C}^5_{n-1}\subseteq
\mathcal{H}_{\rm red}$ and using Lemmas \ref{Cn-1 implies small Cm} and  \ref{Cn-1 implies cm} we
have $\mathcal{C}^5_{m}\subseteq \mathcal{H}_{\rm blue}$. This is a
contradiction.

\bigskip
\noindent \textbf{Case 2. }$n>
\Big\lfloor\frac{3m}{2}\Big\rfloor$.\\
 In this case, $n-1\geq \Big\lfloor\frac{3m}{2}\Big\rfloor.$ Since $$R(\mathcal{C}^5_{n-1},\mathcal{C}^5_{m})=
4(n-1)+\Big\lfloor\frac{m-1}{2}\Big\rfloor<
4n+\Big\lfloor\frac{m-1}{2}\Big\rfloor,$$
we have a copy of  $\mathcal{C}^5_{n-1}$ in $\mathcal{H}_{\rm red}$.
 Applying  Lemma \ref{Cn-1 implies small Cm} for $4\leq m \leq 6$  and  Lemma \ref{Cn-1 implies cm} for $m\geq 7$, we have a blue copy of
$\mathcal{C}^5_{m}$. This
contradiction  completes  the proof. }\end{proof}

\end{document}